\documentclass[11pt,reqno,a4paper]{amsart}

\usepackage{amsmath}
\usepackage{amsthm}
\usepackage{amssymb}

\usepackage{euscript}
\numberwithin{equation}{section}

\newtheorem{theorem}{Theorem}
\newtheorem{proposition}
{Proposition}
\newtheorem{lemma}{Lemma}[section]
\newtheorem{cor}{Corollary}

\theoremstyle{definition}

\newtheorem{remark}{Remark}[section]

\renewcommand{\epsilon}{\varepsilon}

\def\Id{\text{\rm Id}}

\def\N{\mathbb{N}}
\def\Z{\mathbb{Z}}
\def\R{\mathbb{R}}

\subjclass[2010]{Primary: 34D09, Second: 37H15}


\begin{document}

\title[Description of tempered exponential dichotomies]
{Description of tempered exponential dichotomies by admissibility with no Lyapunov norms
}

\begin{abstract}
Tempered exponential dichotomy formulates the nonuniform hyperbolicity for random dynamical systems.
It was described by admissibility of a pair of function classes defined with Lyapunov norms,
For MET-systems (systems satisfying the assumptions of multiplicative ergodic theorem (abbreviated as MET)),
it can be described by admissibility of a pair without a Lyapunov norm.
However,
it is not known how to choose a suitable Lyapunov norms before a tempered exponential dichotomy is given,
and
there are examples
of random systems which are not MET-systems but have a tempered exponential dichotomy.
In this paper we give a description of tempered exponential dichotomy
for general random systems, which may not be MET-systems,
purely by measurable admissibility of three pairs of function classes with no Lyapunov norms.
Further, restricting to the MET-systems,
we obtain a simpler description of only one pair with no Lyapunov norms.
Finally, we use our results to prove the roughness of tempered exponential dichotomies
for parametric random systems
and give a H\"older continuous dependence of the associated projections on the parameter.
\end{abstract}

\author{Davor Dragi\v cevi\' c}
\address{Department of Mathematics, University of Rijeka, Croatia}
\email[Davor Dragi\v cevi\' c]{ddragicevic@math.uniri.hr}

\author{Weinian Zhang}
\address
{School of Mathematics,
    Sichuan University,
    Chengdu, Sichuan 610064, P.R.China }
\email[Weinian Zhang]{matzwn@126.com, wnzhang@scu.edu.cn}

\author{ Linfeng Zhou}
\thanks{Corresponding author: Linfeng Zhou (zlf63math@163.com, zlf63math@scu.edu.cn)}
\address
{School  of Mathematics,
    Sichuan University,
    Chengdu, Sichuan 610064, P.R.China}
\email[Linfeng Zhou]{zlf63math@163.com, zlf63math@scu.edu.cn}
\maketitle

\parskip 0.2cm


\section{Introduction}

Consider the homogeneous linear differential or difference equation
\begin{eqnarray}
\label{lnde}
x'(t)\!\!\!&=&\!\!\!A(t)x(t), \quad t\in  J,\ \ \  {\rm or}
\\
\label{lnde-dis}
x(n+1)\!\!\!&=&\!\!\!A(n)x(n), \quad n\in J,
\end{eqnarray}
on  a Banach space $X$, where $A(s)$ is a  linear operator on $X$ for each $s\in J$ and
$J$ is either
$\mathbb{T}$ ($\mathbb{R}$ or $\mathbb{Z}$), $\mathbb{T}_{+}$ or $\mathbb{T}_{-}$.
These equations remain attractive in research because they are linearization of
nonlinear systems along a specific orbit and some of its basic problems are not answered well.
One of those problems is exponential dichotomy,
which formulates the hyperbolicity of the equilibrium at the origin for Eq.(\ref{lnde}) or Eq.(\ref{lnde-dis}).
A significant question is: How to give existence of an exponential dichotomy?
In addition to the method of finding spectral gaps,
for example,
from the spectrum of the linear part of autonomous systems
(\cite{Henry-book-1981}),
the distribution of Sacker-Sell spectrum of nonautonomous systems (\cite{Sacker-Sell-1994})
and Lyapunov exponents of linear cocycles (\cite{Arnold,Lian-Lu}),
another method is input-output test with admissibility of function classes.
A pair of function classes $(\mathcal{F},\mathcal{E})$ is said to be {\it admissible}
for  Eq.(\ref{lnde}) (resp. Eq.(\ref{lnde-dis}))
if for each input function $f\in\mathcal{F}$,
the equation
\begin{eqnarray}\label{non-lnde}
x'(t)\!\!\!&=&\!\!\!A(t)x+f(t), \quad t\in  J,
\\
({\rm resp.} \ \ \
\label{non-lnde-dis}
x(n+1)\!\!\!&=&\!\!\!A(n)x(n)+f(n+1), \quad n\in J,)
\end{eqnarray}
has a solution in $\mathcal{E}$.
For convenience,
we simply call $\mathcal{F}$ and $\mathcal{E}$ an {\it input class} and an {\it output class} respectively.
The concept of admissibility is of special significance because it can be used to find solutions that belong to a  special class
(e.g., periodic solution, homoclinic solution or heteroclinic solution)
for nonlinear systems, particularly in connection with  invariant manifolds
(see~\cite{Arnold,Duan-Lu-Schmalfuss,Mohammed-Zhang-Zhao,Lian-Lu}) and
an analysis of bifurcations
 (see~\cite{Palmer,Hale-Lin,Gruendler,Zhu-Zhang}).
It is also useful in the study of the persistence of hyperbolic behaviour under small linear perturbations
(see e.g. \cite{Palmer2,BVDbook}).

The study of the relationship between dichotomy and admissibility
has a long history, which goes back to Perron's pioneering work \cite{Perro-1930}
for ODEs.
In 1954, Ma${\rm\breve{i}}$zel$'$ (\cite{Maizel-1954}) proved that
Eq.(\ref{lnde}) on $\mathbb{R}^{d}$, where $A\colon \mathbb{R}_+\to M_d$ (the space of all real $d\times d$ matrices)
is  a bounded  and continuous function,
has an exponential dichotomy on $\mathbb{R}_{+}$
if and only if
for each input function $f\in C_{b}(\mathbb{R}_{+})$  Eq.(\ref{non-lnde})
has a solution in $C_{b}(\mathbb{R}_{+})$, the class of
all bounded continuous functions defined on $\mathbb{R}_{+}$ and valued in $\mathbb{R}^d$.
A landmark work~\cite{Massera-Schaffer-book-1966} was made by
Massera and Sch\"affer,
who considered Eq.\eqref{lnde} with $J=\mathbb R_+$ on a Banach space
and
proposed an axiomatic approach for the construction of possible pairs of input and output classes with the property that the admissibility with respect to those pairs is equivalent to the notion of an exponential dichotomy.
In \cite{Daleckij-Krein-book-1974}, Daleckij and Krein
discussed Eq.\eqref{lnde} on a Banach space with $J=\mathbb R$ such that $\sup_{t\in \mathbb{R}}\int_{t}^{t+1}\|A(\tau)\|d\tau<+\infty$
and proved that
Eq.\eqref{lnde} has an exponential dichotomy on $\mathbb{R}$
if and only if the pair $(C_{b}(\mathbb{R},C_{b}(\mathbb{R}))$
is admissible with the uniqueness, i.e.,
the output solution $x$ given in the above definition of admissibility is unique.
For convenience, we call the admissibility with the uniqueness the {\it proper admissibility} simply.
Additional related results were obtained by Coppel~\cite{Coppel-book-1978} in the case that $A$ is not necessarily a bounded function. In addition, Coppel (\cite[Proposition 2.]{Coppel-book-1978}) dealt with the case that
Eq.\eqref{lnde} does not even have the so-called bounded growth property.
An appropriate admissibility conditions under which Eq.\eqref{lnde} has an exponential dichotomy on both $\mathbb{R}_{+}$ and $\mathbb R_{-}$ (but not necessarily on $\mathbb R$)
were given by Pliss~\cite{Pliss-1977} and Palmer~\cite{Palmer88}.
For related results that deal with the case of noninvertible dynamics, we refer to~\cite{Henry-book-1981, Lin86, Rodr-R95, ZWN95, Chicone-Latushkin-book-1999,BVDbook}. We emphasize that these results correspond to the case that
the operators $A(t)$ in Eq.\eqref{lnde}
are not necessarily bounded, which is typical in the context of parabolic equations.

This paper focuses on discrete-time systems.
In 1967,
Coffman and Sch\"affer (\cite{Coffman-Schaffer-1967}) considered Eq.(\ref{lnde-dis})
on a Banach space $X$ without the invertibility assumption for operators $A(n)$.
They proved that system (\ref{lnde-dis})
has an exponential dichotomy on $\mathbb{Z}_{+}$
if the pair $(\ell^{\infty}(\mathbb{Z}_{+}),\ell^{\infty}(\mathbb{Z}_{+}))$ is admissible,
where $\ell^{\infty}(\mathbb{Z}_{+})$ consists of all bounded sequences in $X$.
Henry (\cite{Henry-book-1981}) obtained the result on $\Z$ with the admissible condition that
 $(\ell^{\infty}(\mathbb{Z}),\ell^{\infty}(\mathbb{Z}))$ is properly admissible.
For more recent results, we refer to~\cite{Huy-Minh-2001,megan-sasu-sasu-2003-DCDS,sasu-sasu-06-01,C-Preda-Onofrei-2018}.
For nonautonomous systems, it is of interest to look for weaker concepts of exponential dichotomy
in which the rates of contraction (along the stable direction) and  of expansion (along the unstable direction)
are allowed to depend on time.
This yields a notion of {\it nonuniform exponential dichotomy} (\cite{Barreira-Valls-book-2008}),
in which the bound depends on the initial time.
Barreira, Dragi${\rm\check{c}}$evi\'c and Valls (\cite{Barreira-Dragicevic-Valls-2014-02})
proved that
Eq.(\ref{lnde-dis}) on $X$
has a nonuniform exponential dichotomy on $\mathbb{Z}$
with a bound being an exponential function
if there is a sequence of norms $(|\cdot|_{n})_{n\in \mathbb{Z}}$ of $X$ (called Lyapunov norms)
such that the pair $(\ell^{\infty}(\mathbb{Z}_{+},\sup_{n\in \mathbb{Z}}|\cdot|_{n}),\ell^{\infty}(\mathbb{Z}_{+},\sup_{n\in \mathbb{Z}}|\cdot|_{n}))$ is properly admissible,
where $\ell^{\infty}(\mathbb{Z}_{+},\sup_{n\in \mathbb{Z}}|\cdot|_{n})$ consists
of all sequences $(x(n))_{n\in \mathbb{Z}}$ in $X$ such that $\sup_{n\in \mathbb{Z}}|x(n)|_{n}<\infty$.
Related results (\cite{Barreira-Dragicevic-Valls-2014-01,ZhouLuZhang17JDE}) were also obtained for continuous-time systems with Lyapunov norms.
We notice that
it is difficult to construct Lyapunov norms before one already knows that the system has a nonuniform exponential dichotomy.
An effort with no Lyapunov norms was made to Eq.(\ref{lnde-dis}) on $\mathbb{R}^{d}$ in \cite{ZhouZhang16JFA},
where
under an additional assumption on the weighted-stable subspaces and the weighted-unstable subspaces
two admissible pairs of function classes are employed to describe nonuniform exponential dichotomies
on $\mathbb{Z}$, $\mathbb{Z}_{+}$ and $\mathbb{Z}_{-}$ separately
or
on both $\mathbb{Z}_{+}$ and $\mathbb{Z}_{-}$ simultaneously
with a bound being an exponential function.
Up to now,
it is open to describe nonuniform exponential dichotomies
by admissibility of function classes
with neither a Lyapunov norm nor the additional assumption of \cite{ZhouZhang16JFA}.

Such an investigation on description of dichotomy by admissibility
was also extended to
cocycles.
More precisely,
some results (see e.g. \cite{Chow-Leiva-1995,LSa,MSS,PPC,Sasu-Sasu-2016})
were obtained for uniform exponential dichotomies of skew-product flows,
each of which is a strongly continuous linear cocycle over a continuous flow.
For random dynamical systems,
each of which is a strongly measurable linear cocycle over a measurable flow,
another type of nonuniform exponential dichotomy, called {\it tempered exponential dichotomy},
was dealt with (see e.g. \cite{Gundlach-1995,Cong-book-1997,Arnold,Lian-Lu}),
where `tempered' means that the bound of dichotomy varies subexponentially along each orbit of the measurable flow
in the base space.
In 2016 Barreira, Dragicevic and Valls (\cite{Barreira-Dragicevic-Valls-2016}) proved that
tempered exponential dichotomy can be described by admissibility
of an appropriate pair of function classes defined with a Lyapunov norm,
which answers to the question about equivalent relationship raised in Remark 1 of \cite{ZhouLuZhang13JDE}.
Later,
Alhalawa and Dragi\v cevi\' c (\cite{Alhalawa-Dragicevic-2017}) gave
a description of tempered exponential dichotomies
by an admissibility with no Lyapunov norms
for MET-systems with a compact generator,
where MET-system is a short name
for random dynamical systems satisfying the assumptions of the Multiplicative Ergodic Theorem (abbreviated as MET)
given in \cite{Oseledets,Arnold,Lian-Lu,Gonzalez-Tokman-Quas-2014}.
However,
it is not known how to choose a suitable Lyapunov norms before a tempered exponential dichotomy is given,
and
an example given in \cite[pp.4044-4045]{ZhouLuZhang13JDE} (also seen in our Remark \ref{re-examp}) shows that
there exists random dynamical systems which are not MET-systems but have a tempered exponential dichotomy.

In this paper we purely use admissibility of function classes with no Lyapunov norms
to give a complete description of tempered exponential dichotomies
for general random systems on a Banach space $X$, which may not be MET-systems.
In section~\ref{P},  we recall the concept of tempered exponential dichotomy,
define several classes of weighted-bounded sequences, which
are constructed in terms of the original norm of $X$ but with no Lyapunov norms
and will play an important role in our results,
and introduce the concept of measurable admissibility.
In section~\ref{T1},
we prove that the existence of a tempered exponential dichotomy implies admissibility of {\it three pairs}
of function classes.
In section~\ref{T2} we conversely
use the three pairs
to describe the existence of a tempered exponential dichotomy.
More precisely, one of the three pairs
is employed to decompose the phase space into direct sum of the stable subspace and the unstable one,
and the other two
are applied to estimating the decay rate in the stable subspace and the growth rate in the unstable subspace separately.
In comparison with \cite{ZhouZhang16JFA},
which needs two admissible pairs plus an assumption on the weighted-stable subspaces
and the weighted-unstable subspaces,
our Theorem \ref{theorem-admi-texd}
purely uses three admissibility pairs but does not need the assumption
 even if we ignore `random'.
In section~\ref{T3},
restricting to MET-systems,
we obtain a simpler description of only one pair with no Lyapunov norms.
Using an induced system to estimate
the upper (resp. lower) bound of the Lyapunov exponents for the random system restricted on the stable (resp. unstable) subspace in Section~\ref{T3}
 is partially inspired by those developed in~\cite{Dragicevic-Sasu-Sasu} for uniform exponential dichotomy.
However, the induced system constructed in \cite{Dragicevic-Sasu-Sasu}
can not be used in our case because of the nonuniformity.
Section~\ref{T4} is devoted to application.
First, we apply our result obtained in section 4
to giving the roughness of tempered exponential dichotomies
for parametric random systems
and give a H\"older continuous dependence of the associated projections on the parameter.
Then, we apply our approach made in section~\ref{T2} to
giving a corresponding description of
dichotomy
for the deterministic difference equation (\ref{lnde-dis}).


\section{Preliminaries}\label{P}

In this section we recall the concept of tempered exponential dichotomy,
define several classes of weighted-bounded sequences, and introduce the concept of measurable admissibility.


\subsection{Tempered exponential dichotomy}
 \

Throughout this paper,
$\mathbb{Z}_{+}$ (resp. $\Z_{-}$)
denotes the set of nonnegative (resp. nonpositive) integers,
$\N$ denotes the set of positive integers,
$(\Omega, \mathcal{F}, \mathbb{P})$ denotes a measure space, $X=(X, |\cdot |)$ is a separable
Banach space,
and $\mathcal{L}(X)$ denotes the Banach space of all bounded linear
operators on $X$.
Furthermore, let $\theta \colon \Omega \to \Omega$ be an invertible measurable transformation that preserves $\mathbb P$.
This means that $\mathbb P(\theta^{-1}(E))=\mathbb P(E)$ for each $E\in \mathcal F$.
A map $\Phi \colon \mathbb{Z}_+\times \Omega \to \mathcal{L}(X)$ is said to be a  \emph{(strongly) measurable linear cocycle}
over $\theta$ (\cite{Arnold})
if
\begin{itemize}
\item $\Phi(0,\omega)=\Id$ for $\omega \in \Omega$, where $\Id$ denotes the identity operator on $X$;
\item $\Phi(n+m,\omega)=\Phi(m,\theta^n (\omega))\Phi(n,\omega)$ for $\omega \in \Omega$ and $n, m\in \mathbb{Z}_{+}$;
\item the map $\omega \mapsto \Phi(n,\omega)$ is strongly measurable for any fixed $n \in \mathbb{Z}_{+}$, i.e. $\omega \mapsto \Phi(n,\omega)x$ is measurable for each $x\in X$.
\end{itemize}
A \emph{ linear random dynamical system} (\cite{Arnold}) on $X$ over $\theta$ is a measurable map
$$
\phi: \mathbb{Z}_+ \times \Omega \times X \to X, \quad (n, \omega, x) \mapsto \phi(n, \omega, x),
$$
such that the map  $(n, \omega) \mapsto \phi(n,\omega,\cdot) \in \mathcal L(X)$
 forms a linear cocycle over $\theta$.
The measurability of $\phi$ implies that the induced linear cocycle is strongly measurable.
Conversely, it is clear that for a given measurable linear cocycle $\Phi$ over $\theta$, the map
$\mathbb{Z}_+ \times \Omega \times X \ni (n, \omega, x) \mapsto \Phi(n, \omega)x$
defines a linear random dynamical system on $X$ over $\theta$. Consequently,
from now on we will use ``linear random dynamical system" and  ``measurable linear cocycle" synonymously.
Let $\Phi$ be a measurable linear cocycle over $\theta$. Set
\[
A(\omega):=\Phi(1,\omega), \quad \omega \in \Omega.
\]
Then
$\omega \mapsto A(\omega)x$ is measurable for each fixed $x\in X$.
We say that $A$ is the \emph{generator} of the cocycle $\Phi$.
 A random variable
 $K\colon \Omega \to (0,+\infty)$ is said to be \emph{tempered}
 with respect to $\theta$ (\cite{Arnold})
  if
\begin{equation}\label{temp}
  \lim_{n\rightarrow
  \pm\infty}\frac{\ln(K(\theta^{n}\omega))}{n}=0, \quad \text{for $\mathbb P$-a.e. $\omega \in \Omega$.}
\end{equation}

We now recall the concept of tempered exponential dichotomy (see~\cite{Arnold,Lian-Lu}).
Let  $\Phi$ be a measurable linear cocycle over $\theta$. We say that $\Phi$ has a \emph{tempered exponential dichotomy}
if there exists a strongly measurable map $\Pi^{s}: \widetilde{\Omega}\rightarrow \mathcal{L}(X)$
defined on a $\theta$-invariant subset $\widetilde{\Omega}\subset \Omega$ of full measure such that:
\begin{itemize}
\item[(i)] $\Pi^s(\omega)$ is a projection on $X$ for $\omega \in \widetilde{\Omega}$;
\item[(ii)]  for all $\omega\in \widetilde{\Omega}$ and $n\in \mathbb{Z}_{+}$,
\begin{equation}\label{inv} \Pi^{s}(\theta^{n}\omega)\Phi(n,\omega)=\Phi(n,\omega)\Pi^{s}(\omega);\end{equation}
\item[(iii)]
 the restriction
 $\Phi(n,\omega)|_{\mathcal{R}(\Pi^{u}(\omega))}$
to the range $\mathcal{R}(\Pi^{u}(\omega))$ of $\Pi^{u}(\omega)$ is
an isomorphism from $\mathcal{R}(\Pi^{u}(\omega))$ onto
$\mathcal{R}(\Pi^{u}(\theta^{n}\omega))$
 for each $\omega\in \widetilde{\Omega}$ and $n\in \mathbb{Z}_{+}$,
 and $\Phi(n,\omega)\Pi^{u}(\omega)$ is strongly measurable on $\omega$
 for $n\in \mathbb{Z}_{-}$,
 where $\Pi^{u}(\omega):=\Id-\Pi^{s}(\omega)$ and
  \[\Phi(n,\omega):=(\Phi(\theta^{n}\omega, -n)|_{\mathcal{R}(\Pi^{u}(\theta^{n}\omega))})^{-1} \quad  \text{for $n\in \mathbb{Z}_{-}$;}\]
\item[(iv)]
 there is a $\theta$-invariant random variable $\alpha:\Omega\rightarrow (0,+\infty)$ (i.e., $\alpha(\theta\omega)=\alpha(\omega)$ for $\omega \in \Omega$) and
 a tempered (with respect to $\theta$) random variable
 $K:\Omega\rightarrow (0,+\infty)$
such that
 \begin{eqnarray}
 \label{ted1}
 &&\|\Phi(n,\omega)\Pi^{s}(\omega)\| \leq
 K(\omega)e^{-\alpha(\omega)n}, \quad
 \forall~ \omega\in \widetilde{\Omega},n\in \mathbb{Z}_{+},
\\
\label{ted2}
 &&\|\Phi(n,\omega)\Pi^{u}(\omega)\| \leq
 K(\omega)e^{\alpha(\omega)n}, \quad
 \  \forall~  \omega\in\widetilde{\Omega},n\in \mathbb{Z}_{-}.
\end{eqnarray}
\end{itemize}

For convenience, we call $\alpha \colon \Omega \to (0, +\infty)$ and $K\colon \Omega \to (0, +\infty)$
the exponent and the bound of the tempered exponential dichotomy respectively.


\subsection{Some classes of weight-bounded sequences}
 \

We introduce some classes of weighted-bounded sequences, which are two-sided sequences in $X$.
Let $\beta, K \colon \Omega \to (0, +\infty)$ be  random variables. For $\omega \in \Omega$, let
\[
\ell^\infty_{1/K(\omega), \beta(\omega)}(\Z):=\bigg \{f=(f(n))_{n\in \Z}\subset X: \sup_{n\in \Z} \big (K(\theta^n \omega)e^{-\beta(\omega)n}|f(n)|\big )<+\infty \bigg \}.
\]
Then, $\ell^\infty_{1/K(\omega), \beta(\omega)}(\Z)$ is a Banach space equipped with
the norm
\[
\| f\|_{\ell^\infty_{1/K(\omega), \beta(\omega)}(\Z)}:= \sup_{n\in \Z} \big (K(\theta^n \omega)e^{-\beta(\omega)n}|f(n)|\big ), \quad f\in \ell^\infty_{1/K(\omega), \beta(\omega)}(\Z).
\]
Similarly, let
\[
\ell^\infty_{|1/K(\omega), \beta(\omega)|}(\Z):=\bigg \{f=(f(n))_{n\in \Z}\subset X: \sup_{n\in \Z} \big (K(\theta^n \omega)e^{-\beta(\omega) |n|}|f(n)|\big )<+\infty \bigg \},
\]
which is a Banach space equipped with the norm
\[
\| f\|_{\ell^\infty_{|1/K(\omega), \beta(\omega)|}(\Z)}:= \sup_{n\in \Z} \big (K(\theta^n \omega)e^{-\beta(\omega) |n|}|f(n)|\big ), \quad f\in \ell^\infty_{|1/K(\omega), \beta(\omega)|}(\Z).
\]
In the particular case $K\equiv 1$, we use $\ell^\infty_{1, \beta (\omega)}(\Z)$ and $\ell^\infty_{|1, \beta (\omega)|}(\Z)$
to denote $\ell^\infty_{1/K(\omega), \beta(\omega)}(\Z)$
 and $\ell^\infty_{|1/K(\omega), \beta(\omega)|}(\Z)$ respectively.
Furthermore, for a fixed $\theta$-invariant measurable set $\widetilde \Omega \subset \Omega$,
we consider
\begin{eqnarray*}
\ell_{1/K,\beta}^{\infty}(\widetilde{\Omega}\times \mathbb{Z})
 :=\{\!\!\!\!\!&&\!\!\!\!\! f:\widetilde{\Omega}\times \mathbb{Z}\rightarrow X:
      f(\omega,\cdot)\in \ell^\infty_{1/K(\omega), \beta(\omega)}(\Z), \forall\ \omega  \in \widetilde{\Omega}
      \},
\\
\ell_{|1/K,\beta|}^{\infty}(\widetilde{\Omega}\times \mathbb{Z})
 :=\{\!\!\!\!\!&&\!\!\!\!\! f:\widetilde{\Omega}\times \mathbb{Z}\rightarrow X:
      f(\omega,\cdot)\in \ell^\infty_{|1/K(\omega), \beta(\omega)|}(\Z), \forall\ \omega\in \widetilde{\Omega}
      \}.
\end{eqnarray*}


\subsection{Measurable admissibility}
 \

Let $\Phi$ be a measurable linear cocycle over $\theta$
with the generator $A$
and
let $\widetilde \Omega \subset \Omega$ be a $\theta$-invariant measurable set. We say that the pair $(\ell_{1/K,\beta}^{\infty}(\widetilde{\Omega}\times \mathbb{Z}),\ell_{1,\beta}^{\infty}(\widetilde{\Omega}\times \mathbb{Z}))$ is \emph{properly admissible} for $\Phi$ if
for each $f\in \ell_{1/K,\beta}^{\infty}(\widetilde{\Omega}\times \mathbb{Z})$
there exists a unique $x\in \ell_{1,\beta}^{\infty}(\widetilde{\Omega}\times \mathbb{Z})$
such that
\begin{eqnarray}\label{madm}
x(\omega, n+1)=A(\theta^{n}\omega)x(\omega, n)+f(\omega, n+1), \quad \forall\ n\in \Z, \omega\in \widetilde{\Omega}.
\end{eqnarray}
By the definitions of $\ell_{1/K,\beta}^{\infty}(\widetilde{\Omega}\times \mathbb{Z})$ and $\ell_{1,\beta}^{\infty}(\widetilde{\Omega}\times \mathbb{Z})$,
the proper admissibility of $(\ell_{1/K,\beta}^{\infty}(\widetilde{\Omega}\times \mathbb{Z}),\ell_{1,\beta}^{\infty}(\widetilde{\Omega}\times \mathbb{Z}))$ for $\Phi$ is equivalent to saying that
for every $\omega \in \widetilde \Omega$
and every
sequence $f=(f(n))_{n\in \Z}\in \ell^\infty_{1/K(\omega), \beta(\omega)}(\Z)$
there exists a unique $x=(x(n))_{n\in \Z}\in \ell^\infty_{1, \beta (\omega)}(\Z)$ such that
\begin{eqnarray}
 x(n+1)=A(\theta^{n}\omega)x(n)+f(n+1),
\quad \forall\ n\in \Z.
\label{eqn-NH}
\end{eqnarray}
In addition, we say that the pair $(\ell_{1/K,\beta}^{\infty}(\widetilde{\Omega}\times \mathbb{Z}),\ell_{1,\beta}^{\infty}(\widetilde{\Omega}\times \mathbb{Z}))$ is \emph{measurably properly admissible} for $\Phi$ if
the pair $(\ell_{1/K,\beta}^{\infty}(\widetilde{\Omega}\times \mathbb{Z}),\ell_{1,\beta}^{\infty}(\widetilde{\Omega}\times \mathbb{Z}))$
is properly admissible and
 for each $f\in \ell_{1/K,\beta}^{\infty}(\widetilde{\Omega}\times \mathbb{Z})$
measurable on $\omega$ (i.e., $\omega \mapsto f(\omega, n)$ is measurable for each $n\in \Z$)
the unique $x\in \ell_{1,\beta}^{\infty}(\widetilde{\Omega}\times \mathbb{Z})$ given by the proper admissibility is also measurable on $\omega$.
In an analogous manner, one can introduce the notion of (measurable) proper admissibility of the pair  $(\ell_{|1/K,\beta|}^{\infty}(\widetilde{\Omega}\times \mathbb{Z}),\ell_{|1,\beta|}^{\infty}(\widetilde{\Omega}\times \mathbb{Z}))$.


\section{Tempered dichotomy implies measurably admissibility}\label{T1}

In this section, we show that the tempered exponential dichotomy implies measurable proper admissibility of several pairs.

\begin{theorem}
\label{theorem-texd-admi}
 Suppose that $\Phi$
 is a measurable linear cocycle over $\theta$ with the generator $A$
  and
 has a tempered exponential dichotomy on $\widetilde{\Omega}\subset \Omega$
 with the exponent $\alpha \colon \Omega \to (0, +\infty)$
 and bound $K\colon \Omega \to (0, +\infty)$.
 Then for any $\theta$-invariant random variable $\beta \colon \Omega \to \R$
with the property $\beta (\omega)\in (-\alpha (\omega), \alpha (\omega))$ for $\omega \in \Omega$,
the followings hold:
\\
{\bf (i)}
the pair $(\ell_{1/K,\beta}^{\infty}(\widetilde{\Omega}\times \mathbb{Z}),\ell_{1,\beta}^{\infty}(\widetilde{\Omega}\times \mathbb{Z}))$
 is measurably properly admissible
 for $\Phi$,
 and additionally
 there is a $\theta$-invariant positive random variable $\gamma \colon \Omega \to (0, +\infty)$ such that
 for $\omega \in \widetilde \Omega$ and $f\in \ell_{1/K(\omega),\beta(\omega)}^{\infty}(\mathbb{Z})$,
\begin{equation}\label{bound}
\|x\|_{\ell_{1,\beta(\omega)}^{\infty}(\mathbb{Z})}\leq \gamma(\omega)\|f\|_{\ell_{1/K(\omega),\beta(\omega)}^{\infty}(\mathbb{Z})},
\end{equation}
where
$x\in \ell_{1,\beta(\omega)}^{\infty}(\mathbb{Z})$ is the (unique) solution  of~\eqref{eqn-NH};
\\
{\bf (ii)}
 the pair $(\ell_{|1/K,\beta|}^{\infty}(\widetilde{\Omega}\times \mathbb{Z}),\ell_{|1,\beta|}^{\infty}(\widetilde{\Omega}\times \mathbb{Z}))$
 is measurably properly admissible for $\Phi$, and additionally
 there is a $\theta$-invariant positive random variable $\widetilde{\gamma} \colon \Omega \to (0, +\infty)$ such that for $\omega \in \widetilde \Omega$ and $f\in \ell_{|1/K(\omega),\beta(\omega)|}^{\infty}(\mathbb{Z})$,
\begin{equation}\label{bound-2}
\|x\|_{\ell_{|1,\beta(\omega)|}^{\infty}(\mathbb{Z})}
 \leq \widetilde{\gamma}(\omega)\|f\|_{\ell_{|1/K(\omega),\beta(\omega)|}^{\infty}(\mathbb{Z})},
\end{equation}
where $x\in \ell_{|1,\beta(\omega)|}^{\infty}(\mathbb{Z})$ is the (unique) solution of~\eqref{eqn-NH}.
\end{theorem}

{\bf Proof.}
Without any loss of generality, we may assume that~\eqref{temp}
holds for each $\omega \in \widetilde{\Omega}$. Indeed, otherwise we can replace $\widetilde{\Omega}$ with $\widetilde{\Omega}\cap \big ( \bigcap_{n\in \Z} \theta^n(\Omega')\big )$, where $\Omega'$ is a full measure set such that~\eqref{temp} holds for each $\omega \in \Omega'$.

For $\omega \in \widetilde{\Omega}$ and $n\in \Z$, let
\begin{equation}
G(n,\omega):= \left\{\begin{array}{ll}
  \Phi(n,\omega)\Pi^{s}(\omega),
  & n\ge 0;
\\
 -\Phi(n,\omega)\Pi^{u}(\omega),
 & n<0,
\end{array}\right.
\label{GGG-Temp}
\end{equation}
where $\Pi^{s}(\omega)$ for $\omega \in \widetilde{\Omega}$ is the projection associated to the tempered exponential dichotomy for $\Phi$
and $\Pi^{u}(\omega):=\Id-\Pi^{s}(\omega)$.
For convenience, we also call $G(n,\omega)$ the Green function
corresponding to the tempered exponential dichotomy as the one corresponding to the uniform exponential dichotomy for deterministic systems
in \cite[p.167]{Daleckij-Krein-book-1974}.

By~\eqref{ted1} and~\eqref{ted2}, we have
\begin{equation}\label{green}
\|G(n,\omega)\|\leq K(\omega)e^{-\alpha(\omega)|n|}, \ \forall \ (\omega, n)\in  \widetilde{\Omega} \times \mathbb{Z}.
\end{equation}
For $f:=(f(n))_{n\in \mathbb{Z}} \in \ell_{1/K(\omega),\beta(\omega)}^{\infty}(\mathbb{Z})$, we have (using~\eqref{green}) that
\begin{equation}\label{de}
\begin{split}
& \bigg|\sum_{k=-\infty}^{+\infty} G(n-k,\theta^{k}\omega)f(k) \bigg|
 \le
 \sum_{k=-\infty}^{+\infty} e^{-\alpha(\omega)|n-k|}e^{\beta(\omega)k}
 \|f \|_{1/K(\omega),\beta(\omega)}\\
 &=
 e^{\beta(\omega)n}\sum_{k=-\infty}^{+\infty} e^{-\alpha(\omega)|n-k|}e^{-\beta(\omega)n}e^{\beta(\omega)k}
  \|f\|_{1/K(\omega),\beta(\omega)}\\
 &=
 e^{\beta(\omega)n} \bigg{\{}\sum_{k=-\infty}^{n} e^{-(\alpha(\omega)+\beta(\omega))(n-k)}
 +\sum_{k=n+1}^{+\infty} e^{-(\alpha(\omega)-\beta(\omega))(k-n)} \bigg{\}}
  \|f \|_{1/K(\omega),\beta(\omega)}\\
 &=
 e^{\beta(\omega)n}\bigg{\{}\frac{1}{1-e^{-(\alpha(\omega)+\beta(\omega))}}
 +\frac{e^{-(\alpha(\omega)-\beta(\omega))}}{1-e^{-(\alpha(\omega)-\beta(\omega))}} \bigg{\}}
  \|f \|_{1/K(\omega),\beta(\omega)},
 \end{split}
\end{equation}
for $n\in \Z$.
We conclude that $x=(x(n))_{n\in \Z}\in \ell_{1,\beta(\omega)}^{\infty}(\mathbb{Z})$,
where
\begin{eqnarray}
\label{admis-solu-x}
 x(n):=\sum_{k=-\infty}^{+\infty} G(n-k,\theta^{k}\omega)f(k), \quad n\in \Z.
\end{eqnarray}
We claim that $x$ is the solution of~\eqref{eqn-NH}.
Indeed,
for any $n\in \mathbb{Z}$ we have that
\[
\begin{split}
 A(\theta^{n}\omega)x(n)
 &=A(\theta^{n}\omega)\sum_{k=-\infty}^{+\infty} G(n-k,\theta^{k}\omega)f(k)
  \\
  &=
  A(\theta^{n}\omega)\sum_{k=-\infty}^{n}\Phi(n-k,\theta^{k}\omega)\Pi^{s}(\theta^{k}\omega)f(k)
  \\
  &\phantom{=}- A(\theta^{n}\omega)\sum_{k=n+1}^{\infty}\Phi(n-k,\theta^{k}\omega)\Pi^{u}(\theta^{k}\omega)f(k)
      \\
  &=
  \sum_{k=-\infty}^{n}\Phi(n+1-k,\theta^{k}\omega)\Pi^{s}(\theta^{k}\omega)f(k)
  \\
  &\phantom{=} -\Pi^{u}(\theta^{n+1}\omega)f(n+1)
  -\sum_{k=n+2}^{\infty}\Phi(n+1-k,\theta^{k}\omega)\Pi^{u}(\theta^{k}\omega)f(k)
       \\
  &=
  \sum_{k=-\infty}^{n+1}\Phi(n+1-k,\theta^{k}\omega)\Pi^{s}(\theta^{k}\omega)f(k)
  -\Pi^{s}(\theta^{n+1}\omega)f(n+1)
  \\
  &\phantom{=}-\Pi^{u}(\theta^{n+1}\omega)f(n+1)-\sum_{k=n+2}^{\infty}\Phi(n+1-k,\theta^{k}\omega)\Pi^{u}(\theta^{k}\omega)f(k)
  \\
  &=
  x(n+1)-f(n+1),
\end{split}
\]
which implies that
\begin{eqnarray*}
  x(n+1)=A(\theta^{n}\omega)x(n)+f(n+1), \quad \forall \ n\in \mathbb{Z}.
\end{eqnarray*}
Hence, $x$ is the solution of~\eqref{eqn-NH}.  Moreover, observe that it follows from~\eqref{de} that \eqref{bound} holds with
the $\theta$-invariant positive random variable
\[
\gamma (\omega):=\frac{1+e^{-(\alpha(\omega)-\beta(\omega))}}{1-e^{-(\alpha(\omega)-|\beta(\omega)|)}}.
\]

In order to show the uniqueness, it is  sufficient  to prove that for $\omega \in \widetilde{\Omega}$,  the only
solution $x=(x(n))_{n\in \Z}\subset X$ of
\begin{equation}\label{hom}
x(n+1)=A(\theta^n \omega)x(n), \quad n\in \Z
\end{equation}
in
 $\ell_{1,\beta(\omega)}^{\infty}(\mathbb{Z})$ is the trivial solution
 $x=0$.
 Let $x:=(x(n))_{n\in \mathbb{Z}}$ be a solution of~\eqref{hom} in $\ell_{1,\beta(\omega)}^{\infty}(\mathbb{Z})$.
For $n\in \Z$, let
\[
x^{s}(n):= \Pi^{s}(\theta^{n}\omega)x(n) \quad \text{and} \quad x^{u}(n):= \Pi^{u}(\theta^{n}\omega)x(n).
\]
Then, we have that
\begin{eqnarray*}
x^{s}(n)&=& \Pi^{s}(\theta^{n}\omega)\Phi(n-m,\theta^{m}\omega)x(m)
\\
  &=&\Phi(n-m,\theta^{m}\omega)\Pi^{s}(\theta^{m}\omega)x(m),
  \quad  \forall \ n\geq m,
\end{eqnarray*}
and similarly
\begin{eqnarray*}
x^{u}(n)
   &=& \Phi(n-m,\theta^{m}\omega)\Pi^{u}(\theta^{m}\omega)x(m), \quad
  \forall\ n\leq m.
\end{eqnarray*}
By~\eqref{ted1} and~\eqref{ted2}, we have that
\begin{eqnarray}
|x^{s}(n)|
 &\le& K(\theta^{m}\omega) e^{-\alpha(\omega)(n-m)}|x(m)|
  ~~\mbox{ for } m\le n,
\label{ss-+}
\\
|x^{u}(n)|
 &\le&
 K(\theta^{m}\omega) e^{-\alpha(\omega)(m-n)}|x(m)|
 ~\,\mbox{ for } m\ge n.
\label{uu--}
\end{eqnarray}
Fix $n\in \Z$.  It follows from (\ref{ss-+}) that
\begin{eqnarray}
 |x^{s}(n)|\leq
  K(\theta^{m}\omega) e^{-\alpha(\omega)n} e^{(\alpha(\omega)+\beta(\omega))m}\|x\|_{\ell_{1,\beta(\omega)}^{\infty}(\mathbb{Z})},
\label{ss-2}
\end{eqnarray}
for $m\le n$.
Since
 $\alpha(\omega)+\beta(\omega)>0$, it follows from~\eqref{temp} that
\[\lim_{m\rightarrow -\infty}K(\theta^{m}\omega) e^{(\alpha(\omega)+\beta(\omega))m}=0.\]
Consequently, by letting $m\to -\infty$ in~\eqref{ss-+} we conclude that $x^s(n)=0$.
Similarly, we obtain that $x^u(n)=0$ and thus $x(n)=x^s(n)+x^u(n)=0$. Since $n\in \Z$ was arbitrary, we conclude that $x=0$.
Therefore, the pair $(\ell_{1/K,\beta}^{\infty}(\widetilde{\Omega}\times \mathbb{Z}),\ell_{1,\beta}^{\infty}(\widetilde{\Omega}\times \mathbb{Z}))$ is properly admissible.

In order to show that this pair is also measurably properly admissible,
let $f\colon \widetilde \Omega \times \Z \to X$
be a measurable map such that $(f(\omega, n))_{n\in \Z}\in \ell_{1/K(\omega),\beta(\omega)}^{\infty}(\mathbb{Z})$ for $\omega \in \widetilde \Omega$.
From the  preceding discussion, we have that the unique solution of~\eqref{madm} is given by
\begin{eqnarray*}
 x(\omega, n):=\sum_{k=-\infty}^{+\infty} G(n-k,\theta^{k}\omega)f(\omega,k), \quad (\omega, n)\in \widetilde \Omega \times \Z.
\end{eqnarray*}
Hence, since $\omega\mapsto f(\omega, n)$ is measurable for any fixed $n\in \mathbb{Z}$,
the map  \[\omega\mapsto \sum_{k=-m}^{m} G(n-k,\theta^{k}\omega)f(\omega, k)\] is measurable  for any fixed $m\in \mathbb{Z}_{+}$. Indeed, this follows directly from the
 strong measurability  of  maps $\omega\mapsto \Phi(n,\omega)$,
 $\omega\mapsto \Pi^{s}(\omega)$ and $\omega\mapsto \Phi(n,\omega)\Pi^{u}(\omega)$.
This implies that for arbitrary $n\in \Z$,   $\omega\mapsto x(\omega, n)$ is measurable
since  it is the pointwise limit of measurable maps
$\omega\mapsto \sum_{k=-m}^{m} G(n-k,\theta^{k}\omega)f(\omega, k)$. We conclude that the pair $(\ell_{1/K,\beta}^{\infty}(\widetilde{\Omega}\times \mathbb{Z}),\ell_{1,\beta}^{\infty}(\widetilde{\Omega}\times \mathbb{Z}))$ is
measurably properly admissible.

We now  prove that the pair
$(\ell_{|1/K,\beta|}^{\infty}(\widetilde{\Omega}\times \mathbb{Z}),\ell_{|1,\beta|}^{\infty}(\widetilde{\Omega}\times \mathbb{Z}))$
 is measurably properly admissible. Similarly to~\eqref{de}, in the case when $\beta(\omega)\ge 0$, for $f:=(f(n))_{n\in \mathbb{Z}} \in \ell_{|1/K(\omega),\beta(\omega)|}^{\infty}(\mathbb{Z})$ we have that
\[
\begin{split}
& \bigg|\sum_{k=-\infty}^{+\infty} G(n-k,\theta^{k}\omega)f(k) \bigg|
 \le
 \sum_{k=-\infty}^{+\infty} e^{-\alpha(\omega)|n-k|}e^{\beta(\omega)|k|}
 \|f \|_{|1/K(\omega),\beta(\omega)|}\\
 &=
 e^{\beta(\omega)|n|}\sum_{k=-\infty}^{+\infty} e^{-\alpha(\omega)|n-k|}e^{-\beta(\omega)|n|}e^{\beta(\omega)|k|}
  \|f \|_{|1/K(\omega),\beta(\omega)|}\\
 &\leq
 e^{\beta(\omega)|n|} \bigg{\{}\sum_{k=-\infty}^{n} e^{-(\alpha(\omega)-\beta(\omega))(n-k)}
 +\sum_{k=n+1}^{+\infty} e^{-(\alpha(\omega)-\beta(\omega))(k-n)} \bigg{\}}
  \|f\|_{|1/K(\omega),\beta(\omega)|}\\
 &=
 e^{\beta(\omega)|n|}\bigg{\{}\frac{1}{1-e^{-(\alpha(\omega)-\beta(\omega))}}
 +\frac{e^{-(\alpha(\omega)-\beta(\omega))}}{1-e^{-(\alpha(\omega)-\beta(\omega))}} \bigg{\}}
  \|f\|_{|1/K(\omega),\beta(\omega)|},
 \end{split}
\]
for each $n\in \Z$, where we used that
 $-\beta(\omega)|n|+\beta(\omega)|k|\leq \beta(\omega)|n-k|$.
Analogously, if  $\beta(\omega)< 0$ we have that
\[
\begin{split}
& \bigg|\sum_{k=-\infty}^{+\infty} G(n-k,\theta^{k}\omega)f(k) \bigg|
 \le
 \sum_{k=-\infty}^{+\infty} e^{-\alpha(\omega)|n-k|}e^{\beta(\omega)|k|}
  \|f\|_{|1/K(\omega),\beta(\omega)|}\\
 &=
 e^{\beta(\omega)|n|}\sum_{k=-\infty}^{+\infty} e^{-\alpha(\omega)|n-k|}e^{-\beta(\omega)|n|}e^{\beta(\omega)|k|}
  \|f\|_{|1/K(\omega),\beta(\omega)|}\\
 &\leq
 e^{\beta(\omega)|n|} \bigg{\{}\sum_{k=-\infty}^{n} e^{-(\alpha(\omega)+\beta(\omega))(n-k)}
 +\sum_{k=n+1}^{+\infty} e^{-(\alpha(\omega)+\beta(\omega))(k-n)} \bigg{\}}
  \|f\|_{|1/K(\omega),\beta(\omega)|}\\
 &=
 e^{\beta(\omega)|n|}\bigg{\{}\frac{1}{1-e^{-(\alpha(\omega)+\beta(\omega))}}
 +\frac{e^{-(\alpha(\omega)+\beta(\omega))}}{1-e^{-(\alpha(\omega)+\beta(\omega))}} \bigg{\}}
  \|f\|_{|1/K(\omega),\beta(\omega)|},
 \end{split}
\]
for each $n\in \Z$, since $-\beta(\omega)|n|+\beta(\omega)|k|\leq -\beta(\omega)|n-k|$.
Hence, (\ref{bound-2}) holds with the $\theta$-invariant positive random variable
\[
\widetilde{\gamma}(\omega)
:=\begin{cases} \frac{1+e^{-(\alpha(\omega)-\beta(\omega))}}{1-e^{-(\alpha(\omega)-\beta(\omega))}}
&  \text{if $\beta(\omega)\geq  0$}\\
\frac{1+e^{-(\alpha(\omega)+\beta(\omega))}}{1-e^{-(\alpha(\omega)+\beta(\omega))}}
  & \text{if $\beta(\omega)<  0$.}
\end{cases}
\]
For the uniqueness and measurability,  one can argue in the same manner as
we did above for the  pair $(\ell_{1/K,\beta}^{\infty}(\widetilde{\Omega}\times \mathbb{Z}),\ell_{1,\beta}^{\infty}(\widetilde{\Omega}\times \mathbb{Z}))$.
This completes the proof of the theorem.
\hfill$\square$

In next section,
 we will conversely
use the three pairs in Theorem~\ref{theorem-texd-admi}
to prove the existence of tempered exponential dichotomy
for measurable cocycles which may not satisfy the assumptions of MET.


\section{Converse without MET}\label{T2}

The main result on the converse reads:

\begin{theorem}
\label{theorem-admi-texd}
 Suppose that $\Phi$
 is a measurable linear cocycle over $\theta$
with the generator $A$. Furthermore,
assume that there exist a $\theta$-invariant measurable subset $\widetilde{\Omega}\subset \Omega$ of full measure,
 a $\theta$-invariant random variable $\beta \colon \Omega \to (0, +\infty)$
 and a tempered (with respect to $\theta$) random variable $K\colon \Omega \to (0, +\infty)$ such that
\begin{itemize}
\item[(S1)] for $\lambda=\pm \beta$,
the pair $(\ell_{1/K,\lambda}^{\infty}(\widetilde{\Omega}\times \mathbb{Z}),\ell_{1,\lambda}^{\infty}(\widetilde{\Omega}\times \mathbb{Z}))$
is measurably properly admissible,
and additionally there is a $\theta$-invariant random variable $\gamma_{\lambda}
\colon \Omega \to (0, +\infty)$ such that for $\omega \in \widetilde \Omega$ and $f=(f(n))_{n\in \Z}\in \ell_{1/K(\omega),\lambda(\omega)}^{\infty}(\mathbb{Z})$,
\begin{equation}\label{io}
\|x\|_{\ell_{1,\lambda(\omega)}^{\infty}(\mathbb{Z})}\leq \gamma_{\lambda}(\omega)\|f\|_{\ell_{1/K(\omega),\lambda(\omega)}^{\infty}(\mathbb{Z})},
\end{equation}
where $x=(x(n))_{n\in \Z}$ is the unique solution of~\eqref{eqn-NH} in
 $\ell^\infty_{1, \lambda(\omega)}(\Z)$;

\item[(S2)] the pair  $(\ell_{|1/K,\beta|}^{\infty}(\widetilde{\Omega}\times \mathbb{Z}),\ell_{|1,\beta|}^{\infty}(\widetilde{\Omega}\times \mathbb{Z}))$ is properly admissible.
\end{itemize}
Then,  $\Phi$ has  a tempered exponential dichotomy.
\end{theorem}

Note that the theorem needs three admissible pairs
$(\ell_{1/K,\beta}^{\infty}(\widetilde{\Omega}\times \mathbb{Z}),\ell_{1,\beta}^{\infty}(\widetilde{\Omega}\times \mathbb{Z}))$,
$(\ell_{1/K,-\beta}^{\infty}(\widetilde{\Omega}\times \mathbb{Z}),\ell_{1,-\beta}^{\infty}(\widetilde{\Omega}\times \mathbb{Z}))$
and
$(\ell_{|1/K,\beta|}^{\infty}(\widetilde{\Omega}\times \mathbb{Z}),\ell_{|1,\beta|}^{\infty}(\widetilde{\Omega}\times \mathbb{Z}))$,
as mentioned in the Introduction.

{\bf Proof of Theorem \ref{theorem-admi-texd}.}
The proof will be divided into several steps.
For any fixed $\omega\in \widetilde{\Omega}$ and $\lambda=\pm \beta$,
let
$L_{\lambda,\omega} \colon \mathcal D(L_{\lambda, \omega})\subset \ell_{1,\lambda(\omega)}^{\infty}(\mathbb{Z}) \to \ell_{1/K(\omega),\lambda(\omega)}^{\infty}(\mathbb{Z})$ be the linear operator
defined by
\begin{eqnarray*}
 (L_{\lambda,\omega}x)(n):=x(n)-A(\theta^{n-1}\omega)x(n-1), \quad  \ n\in \mathbb{Z},
\end{eqnarray*}
for each $x\in \mathcal D(L_{\lambda, \omega})$,
where $D(L_{\lambda, \omega})$ denotes the domain of $L_{\lambda,\omega}$,
that is,
$\mathcal D(L_{\lambda, \omega}):=\{x\in \ell_{1,\lambda(\omega)}^{\infty}(\mathbb{Z}): L_{\lambda,\omega}x \in \ell_{1/K(\omega),\lambda(\omega)}^{\infty}(\mathbb{Z})\}
$.

\begin{lemma}\label{CO}
$L_{\lambda,\omega}$ is a closed linear operator.
\end{lemma}

\begin{proof}[{\bf Proof}]
Let $(x_k)_{k\in \N}$ be a sequence in $\mathcal D(L_{\lambda, \omega})$ such that sequences $(x_k)_{k\in \N}$ and $(L_{\lambda, \omega}x_k)_{k\in \N}$ converge in $\ell_{1,\lambda(\omega)}^{\infty}(\mathbb{Z})$ and $\ell_{1/K(\omega),\lambda(\omega)}^{\infty}(\mathbb{Z})$ respectively. Let
\[
x:=\lim_{k\to \infty} x_k \quad \text{and} \quad y:=\lim_{k\to \infty} L_{\lambda, \omega}x_k.
\]
Then, for any $n\in \Z$ we have
\[
\begin{split}
x(n)-A(\theta^{n-1}\omega)x(n-1) &=\lim_{k\to \infty}(x_k(n)-A(\theta^{n-1}\omega)x_k(n-1)) \\
&=\lim_{k\to \infty}(L_{\lambda, \omega}x_k)(n)\\
&=y(n).
\end{split}
\]
Hence, $x\in \mathcal D(L_{\lambda, \omega})$ and $L_{\lambda, \omega}x=y$. The proof is completed.
\end{proof}

Observe that the assumptions of the theorem imply that $L_{\lambda, \omega}$ is invertible.
 By Lemma~\ref{CO}, $L_{\lambda, \omega}^{-1}\colon \ell_{1/K(\omega),\lambda(\omega)}^{\infty}(\mathbb{Z}) \to \ell_{1,\lambda(\omega)}^{\infty}(\mathbb{Z})$
is closed and thus bounded.

For any $n, k\in \Z$ and $\omega \in \widetilde \Omega$, we define a linear operator $G_{\lambda, \omega}(n, k)\colon X\to X$ in the following way. For $z\in X$, let  $\tilde{f}_{z,k}:=(\tilde{f}_{z,k}(m))_{m\in \mathbb{Z}}$ be the sequence
given by
\begin{equation}
\tilde{f}_{z,k}(m):= \left\{\begin{array}{ll}
  z,
  & m=k,
\\
 0,
 & m\neq k.
\end{array}\right.
\label{def-td-f-zk}
\end{equation}
We now let
\[
G_{\lambda, \omega}(n, k) z:=(L_{\lambda, \omega}^{-1} \tilde{f}_{z,k})(n).
\]
Observe that $G_{\lambda, \omega}(n,k)$ is a bounded linear operator. Indeed, using~\eqref{io}, for any $z\in X$ we have
\[
\begin{split}
e^{-\lambda (\omega)n}|G_{\lambda, \omega}(n, k) z| &=e^{-\lambda (\omega)n}|(L_{\lambda, \omega}^{-1} \tilde{f}_{z,k})(n)|  \\
&\le \lVert L_{\lambda, \omega}^{-1} \tilde{f}_{z,k}\rVert_{\ell_{1,\lambda(\omega)}^{\infty}(\mathbb{Z})}\\
&\le \lVert  L_{\lambda, \omega}^{-1}\rVert \cdot \lVert \tilde{f}_{z,k}\rVert_{\ell_{1/K(\omega),\lambda(\omega)}^{\infty}(\mathbb{Z})}\\
&\le \gamma_\lambda(\omega)K(\theta^k (\omega))e^{-\lambda (\omega)k}|z|,
\end{split}
\]
and thus
\begin{equation}\label{G}
\lVert G_{\lambda, \omega}(n, k)\rVert \le \gamma_\lambda(\omega)K(\theta^k (\omega))e^{-\lambda (\omega)k}e^{\lambda (\omega)n}.
\end{equation}

\begin{lemma}\label{SM}
For each $(n,k)\in \Z\times \Z$, the map $\omega \mapsto G_{\lambda, \omega}(n,k)$ is strongly measurable.
\end{lemma}

\begin{proof}[{\bf Proof}]
Fix $(n,k)\in \Z\times \Z$.
For any fixed $z\in X$, define $f\colon \widetilde \Omega \times \Z \to X$ as
\[
f(\omega, m)=\begin{cases}
z & \text{for $m=k$;}\\
0 & \text{for $m\neq k$,}
\end{cases}
\]
for $(\omega, m)\in \widetilde \Omega \times \Z$. As we have observed, $(f(\omega, m))_{m\in \Z}\in \ell_{1/K(\omega),\lambda(\omega)}^{\infty}(\mathbb{Z})$ for $\omega \in \widetilde \Omega$.
Since the pair $(\ell_{1/K,\lambda}^{\infty}(\widetilde{\Omega}\times \mathbb{Z}),\ell_{1,\lambda}^{\infty}(\widetilde{\Omega}\times \mathbb{Z}))$ is measurably properly admissible, there exists a unique measurable map $x\colon \widetilde{\Omega} \times \Z \to X$ such that
 $(x(\omega, m))_{m\in \Z}\in \ell^\infty_{1, \lambda (\omega)}(\Z)$ for $\omega \in \widetilde \Omega$ and
\[
x(\omega, m+1)=A(\theta^{m}\omega)x(\omega, m)+f(\omega, m+1), \quad \text{for $m\in \Z$.}
\]
It follows that
\[
G_{\lambda, \omega}(n,k)z=x(\omega, n).
\]
Since $\omega \mapsto x(\omega, n)$ is measurable, we conclude that $\omega \mapsto G_{\lambda, \omega}(n,k)z$ is also measurable and the proof of the lemma is completed.
\end{proof}

\begin{lemma}\label{L2}
For any $(n,k)\in \Z\times \Z$,
$G_{\beta, \omega}(n,k)=G_{-\beta, \omega}(n,k)$.
\end{lemma}

\begin{proof}[{\bf Proof}]
For any fixed $z\in X$, let $\tilde{f}_{z,k}:=(\tilde{f}_{z,k}(m))_{m\in \mathbb{Z}}$ be the sequence given by~\eqref{def-td-f-zk}. Let $x:=(x(m))_{m\in \Z}\subset X$ be given by
\[
x(m):=(L_{\beta, \omega}^{-1}\tilde{f}_{z,k}-L_{-\beta, \omega}^{-1}\tilde{f}_{z,k})(m), \quad m\in \Z.
\]
Then,
\begin{equation}\label{lde}
x(m)-A(\theta^{m-1}\omega)x(m-1)=0, \quad \forall m\in \Z.
\end{equation}
Moreover, since $\ell_{1,\beta(\omega)}^{\infty}(\Z)\subset \ell_{|1, \beta(\omega)|}^\infty(\Z)$
and $\ell_{1, -\beta(\omega)}^{\infty}(\Z)\subset \ell_{|1, \beta(\omega)|}^\infty(\Z)$, we have that
$x\in \ell_{|1, \beta(\omega)|}^\infty(\Z)$.
From our assumption that the pair $(\ell_{|1/K,\beta|}^{\infty}(\widetilde{\Omega}\times \mathbb{Z}),\ell_{|1,\beta|}^{\infty}(\widetilde{\Omega}\times \mathbb{Z}))$ is properly admissible,
we conclude from~\eqref{lde} that $x=0$.
Consequently, we have that $x(n)=0$ and thus
\[
G_{\beta, \omega}(n,k)z=L_{\beta, \omega}^{-1}\tilde{f}_{z,k}(n)=L_{-\beta, \omega}^{-1}\tilde{f}_{z,k}(n)=G_{-\beta, \omega}(n,k)z.
\]
\end{proof}

For $\omega \in \widetilde{\Omega}$ and $n\ge m$, let
\[
U_{\omega}(n,m):=\begin{cases}
A(\theta^{n-1}\omega)\cdot\cdot\cdot A(\theta^{m}\omega) & \text{for $n>m$;}\\
\Id & \text{for $n=m$.}
\end{cases}
\]
Observe that $U_{\omega}(n,m)=\Phi(n-m,\theta^{m}\omega)$. In addition, for $\omega \in \widetilde{\Omega}$ and $m\in \Z$, let
\[
\mathcal S_\omega(m):=\bigg \{v\in X: \sup_{n\ge m}|U_\omega(n,m)v| <+\infty \bigg \}.
\]
Finally, let $\mathcal U_\omega(m)$ consist of all $v\in X$ for which there exists a sequence $(x(n))_{n\le m}$ such that $x(m)=v$, $\sup_{n\le m}|x(n)|<+\infty$ and
\[x(n)=A(\theta^{n-1}\omega)x(n-1), \quad  \text{for $n\le m$.}\] Clearly, $\mathcal S_\omega(m)$ and $\mathcal U_\omega(m)$ are subspaces of $X$. We will write $G_\omega(n, k)$ instead of $G_{\beta, \omega}(n,k)=G_{-\beta, \omega}(n,k)$. Furthermore, let
\[
P_{\omega}^{s}(m):=G_{\omega}(m,m), \quad m\in \Z.
\]

\begin{lemma}
\label{lm-SU-subspace}
For any $\omega \in \widetilde{\Omega}$ and $m\in \Z$, $P_{\omega}^{s}(m)$ is a projection on $X$ such that
$\mathcal{R}(P_{\omega}^{s}(m))=\mathcal{S}_{\omega}(m)$ and
$\mathcal{N}(P_{\omega}^{s}(m))=\mathcal{U}_{\omega}(m)$.
\end{lemma}

\begin{proof}[{\bf Proof}]
For any fixed $v\in \mathcal S_\omega(m)$, define $x=(x(n))_{n\in \Z}\subset X$ by
\[
x(n) =\begin{cases}
U_\omega(n,m)v & \text{for $n\ge m$;}\\
0 & \text{for $n<m$.}
\end{cases}
\]
Note that \[ x(n)=A(\theta^{n-1}\omega)x(n-1)+\tilde f_{v,m}(n),\quad  \text{for $n\in \Z$.}\] In addition,
\[
\sup_{n\in \Z}(|x(n)|e^{-\beta(\omega)n})=\sup_{n\ge m}(|x(n)|e^{-\beta(\omega)n})\le e^{\beta (\omega)|m|}\sup_{n\ge m}|x(n)|<+\infty.
\]
Hence, $x\in \ell^\infty_{1, \beta(\omega)}(\Z)$. We conclude that $L_{\beta, \omega}^{-1}\tilde f_{v, m}=x$
and therefore
$P_\omega^{s}(m)v=v$.
Consequently, $\mathcal S_\omega(m) \subset \mathcal{R}(P_{\omega}^{s}(m))$.

Now, choose an arbitrary $v\in X$ and let $y:=L_{-\beta, \omega}^{-1}\tilde f_{v, m}$. Then, $y(n)=U_\omega(n, m)y(m)$ for $n\ge m$. Moreover,
\[
\sup_{n\ge m}|y(n)|\le e^{\beta(\omega)|m|}\sup_{n\ge m}(|y(n)| e^{\beta (\omega)n})\le  e^{\beta(\omega)|m|} \|y\|_{\ell_{1, -\beta(\omega)}^\infty(\Z)}<+\infty.
\]
Thus, $P_\omega^s(m) v=y(m)\in \mathcal S_\omega(m)$. We conclude that $P_{\omega}^{s}(m)$ is a projection onto $S_\omega(m)$.

In addition, if we assume $P_\omega^s(m)v=0$, then $y(m)=0$ and consequently,
\[
-A(\theta^{m-1}\omega)y(m-1)=v \quad \text{and} \quad y(n)=A(\theta^{n-1}\omega)y(n-1), \ \forall n\le m-1.
\]
Since $y\in \ell_{1, \beta (\omega)}^\infty(\Z)$, we have that $v\in \mathcal U_\omega(m)$. Thus, $\mathcal N(P_\omega^s(m))\subset \mathcal U_\omega(m)$. In order to prove that $\mathcal U_\omega(m) \subset \mathcal N(P_\omega^s(m))$, it is sufficient to show that
$\mathcal{S}_{\omega}(m)\cap \mathcal{U}_{\omega}(m)=\emptyset$.
For any fixed $v\in \mathcal{S}_{\omega}(m)\cap \mathcal{U}_{\omega}(m)$, let $(x(n))_{n\le m}\subset X$ satisfy that
$x(m)=v$, $x(n)=A(\theta^{n-1}\omega)x(n-1)$ for $n\le m$ and $\sup_{n\le m}|x(n)|<+\infty$. We define $z=(z(n))_{n\in \Z}\subset X$ as
\[
z(n)=\begin{cases}
U_\omega(n,m)v & \text{for $n\ge m$;}\\
x(n) & \text{for $n<m$.}
\end{cases}
\]
Then, $\sup_{n\in \Z}|z(n)|<+\infty$ and thus $z\in \ell_{|1, \beta(\omega)|}^\infty(\Z)$. Moreover,
\[
z(n)-A(\theta^{n-1}\omega)z(n-1)=0, \quad \forall n\in \Z.
\]
Hence, $z=0$ and thus $z(m)=v=0$. The proof is completed.
\end{proof}

\begin{lemma}\label{inv2}
For $(\omega, m)\in \widetilde{\Omega}\times \Z$,
\[
A(\theta^{m}\omega)\mathcal S_\omega(m)\subset \mathcal S_\omega (m+1) \quad \text{and} \quad A(\theta^{m}\omega)\mathcal U_\omega(m)\subset \mathcal U_\omega (m+1).
\]
\end{lemma}

\begin{proof}[{\bf Proof}]
For any fixed $v^{s}\in \mathcal S_\omega(m)$, $\sup_{n\ge m}|U_\omega(n,m)v^{s}|<+\infty$. Since
$U_\omega(n, m+1)A(\theta^m \omega)v^{s}=U_\omega(n,m)v^{s}$, we have that \[\sup_{n\ge m+1}|U_\omega(n, m+1)A(\theta^m \omega)v^{s} |<+\infty.\]
Consequently, $A(\theta^m \omega)v^{s}\in \mathcal S(m+1)$.

For any fixed $v^{u}\in \mathcal U(m)$, there exists a sequence $(x(n))_{n\le m}\subset X$ such that $x(m)=v^{u}$, $x(n)=A(\theta^{n-1}\omega)x(n-1)$ for $n\le m$ and $\sup_{n\le m}|x(n)|<+\infty$.
Define a sequence $(x'(n))_{n\le m+1}\subset X$ as
\[
x'(n)=\begin{cases}
x(n) & \text{for $n\le m$;}\\
A(\theta^m \omega)v^{u} & \text{for $n=m+1$.}
\end{cases}
\]
Then, $x'(m+1)=A(\theta^m \omega)v^{u}$, $x'(n)=A(\theta^{n-1}\omega)x'(n-1)$ for $n\le m+1$ and $\sup_{n\le m+1}|x'(n)|<+\infty$. Hence, $A(\theta^m \omega)v^{u} \in \mathcal U_\omega(m+1)$.
\end{proof}

\begin{lemma}\label{I}
$
A(\theta^{m}\omega)\rvert_{\mathcal U_\omega(m)}\colon \mathcal U_\omega(m)\to \mathcal U_\omega (m+1)
$
is invertible
for $(\omega, m)\in \widetilde{\Omega}\times \Z$.
\end{lemma}

\begin{proof}[{\bf Proof}]
For any fixed $w\in \mathcal U_\omega (m+1)$, there exists a sequence $(x(n))_{n\le m+1}\subset X$ such that $x(m+1)=w$, $x(n)=A(\theta^{n-1}\omega)x(n-1)$ for $n\le m+1$ and $\sup_{n\le m+1}|x(n)|<+\infty$.
Hence, $x(m)\in \mathcal U_\omega(m)$ and $A(\theta^{m}\omega)x(m)=w$. Therefore, $A(\theta^{m}\omega)\rvert_{\mathcal U_\omega(m)}\colon \mathcal U_\omega(m)\to \mathcal U_\omega (m+1)$ is surjective.

Assume now that $A(\theta^{m}\omega)v=0$ for $v\in  \mathcal U_\omega(m)$.
Since $v\in \mathcal U_\omega(m)$, there exists a sequence  $(x(n))_{n\le m}\subset X$ such that $x(m)=v$, $x(n)=A(\theta^{n-1}\omega)x(n-1)$ for $n\le m$ and
$\sup_{n\le m}|x(n)|<+\infty$. We define $z=(z(n))_{n\in \Z}\subset X$ as
\[
z(n)=\begin{cases}
x(n) & \text{for $n\le m$;}\\
0 & \text{for $n>m$.}
\end{cases}
\]
Then, $z\in \ell_{|1, \beta(\omega)|}^\infty (\Z)$ and $z(n)=A(\theta^{n-1}\omega)z(n-1)$ for $n\in \Z$. Hence, $z=0$ and thus $v=z(m)=0$. We conclude that  $A(\theta^{m}\omega)\rvert_{\mathcal U_\omega(m)}\colon \mathcal U_\omega(m)\to \mathcal U_\omega (m+1)$ is injective and the
proof of the lemma is completed.
\end{proof}

\begin{lemma}
\label{lm-Green-funct}
For $\omega \in \widetilde{\Omega}$,
\[
G_{\omega}(n,m) =\begin{cases}
U_\omega(n,m)P_\omega^s(m) & \text{for $n\ge m$;}
\\
-U_{\omega}(n,m)(\Id-P_{\omega}^{s}(m)) & \text{for $n<m$.}
\end{cases}
\]
\end{lemma}

\begin{proof}[{\bf Proof}]
For any fixed $v\in X$, let  $y:=L_{-\beta, \omega}^{-1}\tilde f_{v, m}$. Then, $G_\omega(n,m)v=y(n)$ for $n\in \Z$. In particular, $P_{\omega}^{s}(m)v=G_\omega(m,m)v=y(m)$.
Since
$
y(m)-v=A(\theta^{m-1}\omega)y(m-1)$ and  $y(n)=A(\theta^{n-1}\omega)y(n-1)$ for $n\neq m$,
the desired claim follows.
\end{proof}

\begin{lemma}\label{EST}
There exists
a $\theta$-invariant random variable $\rho_{\beta} \colon \Omega \to (0, +\infty)$ such that
$\|G_{\omega}(n,m)\|\leq \rho_{\beta}(\omega)K(\theta^{m}\omega)e^{-\beta(\omega)|n-m|}$
for $\omega \in \widetilde{\Omega}$ and  $n,m \in \mathbb{Z}$.
\end{lemma}

\begin{proof}[{\bf Proof}]
By (\ref{G}), the definition of $G_{\omega}$ and Lemma \ref{L2}, we get that
\begin{eqnarray*}
\lVert G_{\omega}(n, m)\rVert
= \lVert G_{-\beta, \omega}(n, m)\rVert
&\leq & \gamma_{-\beta}(\omega)K(\theta^m (\omega))e^{\beta (\omega)m}e^{-\beta (\omega)n}
\\
&=&\gamma_{-\beta}(\omega)K(\theta^m (\omega))e^{-\beta (\omega)(n-m)},
\ \forall \ n\geq m,
\end{eqnarray*}
and
\begin{eqnarray*}
\lVert G_{\omega}(n, m)\rVert
= \lVert G_{\beta, \omega}(n, m)\rVert
&\leq & \gamma_{\beta}(\omega)K(\theta^m (\omega))e^{-\beta (\omega)m}e^{\beta (\omega)n}
\\
&=&\gamma_{\beta}(\omega)K(\theta^m (\omega))e^{-\beta (\omega)(m-n)},
\ \forall \ n< m.
\end{eqnarray*}
Let \[\rho_{\beta}(\omega):=\max\{\gamma_{\beta}(\omega), \gamma_{-\beta}(\omega)\}.\]
Then, $\rho_{\beta}$ is  a positive $\theta$-invariant random variable.
It follows that
\[
\|G_{\omega}(n,m)\|\leq \rho_{\beta}(\omega)K(\theta^{m}\omega)e^{-\beta(\omega)|n-m|},
\]
for all $m, n\in \Z$. This completes the proof of the lemma.
\end{proof}

Let
\[
\Pi^s (\omega):=P_\omega^s(0), \quad \omega \in \widetilde{\Omega}.
\]
We claim that $\Phi$ has a tempered exponential dichotomy on $\widetilde{\Omega}$
with the projection $\Pi^s (\omega)$.
In fact,
by Lemma~\ref{SM}, we obtain that $\omega \mapsto \Pi^s(\omega)$ is strongly measurable
because $P_{\omega}^{s}(0)=G_{\omega}(0,0)$.
By the definitions of $S_\omega(m)$ and $\mathcal U_\omega(n)$,
one can check that
\[
\mathcal S_\omega(n)=\mathcal S_{\theta^n \omega}(0) \quad \text{and} \quad \mathcal U_\omega(n)=\mathcal U_{\theta^n \omega}(0)
\]
for $\omega \in \widetilde{\Omega}$ and $n\ge 0$.
This together with Lemma \ref{lm-SU-subspace} implies that $P_\omega^s(n)=P_{\theta^{n}\omega}^s(0)$. It follows that
\begin{eqnarray}
\label{Pis-Ps-1}
\Pi^s (\theta^{n}\omega)
=P_\omega^s(n).
\end{eqnarray}
By Lemmas \ref{lm-SU-subspace} and \ref{inv2} and the cocycle property of $\Phi$, we have
\begin{eqnarray*}
U_{\omega}(n,0)P_\omega^s(0)=P_\omega^s(n)U_{\omega}(n,0),\ \forall \ n\in \mathbb{Z}_{+}.
\end{eqnarray*}
This together with (\ref{Pis-Ps-1}) implies that ~\eqref{inv} holds because $U_{\omega}(n,0)=\Phi(n,\omega)$ for $n\in \mathbb{Z}_{+}$.
Furthermore,
let $\Pi^u (\omega):=\Id-\Pi^s (\omega)$.
By Lemma~\ref{I} and the cocycle property of $\Phi$, we get that
$\Phi(n,\omega)|_{\mathcal R(\Pi^u(\omega))}:\mathcal R(\Pi^u(\omega))\rightarrow \mathcal R(\Pi^u(\theta^{n}\omega))$
is an isomorphism.
By Lemmas \ref{SM} and \ref{lm-Green-funct},
$\Phi(n,\omega)\Pi^{u}(\omega)$ is strongly measurable on $\omega$
 for $n\in \mathbb{Z}_{-}$.
The property (iv) in the definition of the tempered exponential dichotomy follows readily from Lemmas \ref{lm-Green-funct} and \ref{EST}.
This is proves the claim and therefore Theorem \ref{theorem-admi-texd} is true.
\hfill$\square$

\begin{remark}
\label{re-assum-gamma}
Note that the linear operator $L_{\lambda, \omega}^{-1}\colon \ell_{1/K(\omega),\lambda(\omega)}^{\infty}(\mathbb{Z}) \to \ell_{1,\lambda(\omega)}^{\infty}(\mathbb{Z})$, discussed just after Lemma \ref{CO}, is bounded.
So, in the assumption (S1) of Theorem \ref{theorem-admi-texd},
the existence of a positive variable $\gamma_{\lambda}$ which satisfies
(\ref{io})
is not an additional assumption.
\end{remark}

\begin{remark}
\label{re-examp}
Let $\Omega=[0,1)$,  $\mathcal{B}(\Omega)$ be the Borel $\sigma$-algebra, and $\mathbb{P}$
be the Lebesgue measure.
Then $(\Omega,\mathcal{B}(\Omega),\mathbb{P})$ is a probability space.
For a fixed irrational number $q\in \Omega$, consider
$$
\theta: \Omega\rightarrow \Omega, ~~~\omega\mapsto \omega+q~ ({\rm
mod} ~1).
$$
Let
$\Omega=\bigcup_{i=1}^{+\infty}\Omega_{i}$, where
$\Omega_{i}:=[1-\frac{1}{i},1-\frac{1}{i+1})$.
Let $A:\Omega \rightarrow \mathcal{L}(\mathbb{R}^{3})$ be defined piecewise by the $3\times 3$ matrix
\begin{eqnarray*}
 A(\omega):={\rm diag}\{e^{-(i+1)^{2}+i^{2}}, 1/2, e^{(i+1)^{2}-i^{2}}\},
\ \forall~\omega\in\Omega_{i},~i\in \mathbb{Z}_{1}^{+}.
\end{eqnarray*}
It follows that $A$ is
$(\mathcal{B}(\Omega), \mathcal{L}(\mathbb{R}^{3}))$-measurable and further
$A$ generates a linear random dynamical system on the phase space $X=\mathbb{R}^{3}$.
It is shown in \cite[pp.4044-4045]{ZhouLuZhang13JDE}
that this random dynamical system
has a tempered exponential dichotomy with the projection
$\Pi^{s}(\omega):={\rm diag}\{1,1,0\}$,
exponent $\alpha(\omega)=\ln 2$ and bound $K(\omega)\equiv 1$,
but does not satisfy the assumptions of  MET.
\end{remark}

Theorem \ref{theorem-admi-texd} shows that
admissibility of the three pairs of input and output classes obtained in Theorem \ref{theorem-texd-admi}
implies existence of a tempered exponential dichotomy for general random dynamical systems without the assumptions of MET.
In the following section 5, for MET-systems,
we will give a simper admissibility condition with no Lyapunov norms in the sense
that only one special pair of input and output classes is needed.


\section{Converse via MET}\label{T3}

Throughout this section, we continue to assume that $(\Omega, \mathcal F, \mathbb P, \theta)$ is an invertible $\mathbb P$-preserving dynamical system. In addition, we assume that
\begin{itemize}
\item $\Omega$ is a compact metric space and $\mathcal F$ is the associated Borel $\sigma$-algebra;
\item $\mathbb P$ is ergodic, i.e. $\mathbb P(E)\in \{0, 1\}$ for each $E\in \mathcal F$ such that $\theta^{-1}(E)=E$.
\end{itemize}
In order to recall the version of MET established in~\cite{Gonzalez-Tokman-Quas-2014},
we introduce several concepts.
For a fixed $A\in \mathcal{L}(X)$, the constant
\begin{eqnarray*}
  \kappa_{ic}(A):=\inf\{r: A(B_{X}(0,1))~{\rm has~ a~ finite~ cover~ by~ balls~ of~ radius~ r}\}
\end{eqnarray*}
is called the {\it index of compactness} of $A$, where $B_{X}(0,1)$ denotes the unit ball in $X$ centered at $0$.
Let $\Phi$ be a measurable linear cocycle over $\theta$ with the generator $A$ such that
\begin{equation}\label{intmet}
\int_M \ln^+ \lVert A(\omega)\rVert\, d\mathbb P(\omega)<\infty.
\end{equation}
It follows essentially from the subadditive ergodic theorem (see Lemmas 2.4 and 2.5 of \cite{Gonzalez-Tokman-Quas-2014})
that there exist $-\infty \le \kappa^*\le \lambda^*<\infty$ such that
\begin{eqnarray*}
\kappa^*=\lim_{n\rightarrow \infty}\frac{\ln \kappa_{ic}(\Phi(n,\omega))}{n}
 ~~~{\rm and}~~~
\lambda^*=\lim_{n\rightarrow \infty}\frac{\ln \|\Phi(n,\omega)\|}{n}
\end{eqnarray*}
for $\mathbb P$-a.e. $\omega \in \Omega$,
which are called
the {\it index of compactness}
and
the {\it maximal Lyapunov exponent}
of $\Phi$ (with respect to $\mathbb P$) respectively.
If $\kappa^{*}<\lambda^{*}$, we say that
 $\Phi$ is  {\it quasi-compact}. Observe that when $\Phi$ is compact, i.e. when $A(\omega)$ is a compact operator for each $\omega \in \Omega$, we have that $\kappa^*=-\infty$ since $\kappa_{ic}(\Phi(n, \omega))=0$ for each $n\in \N$ and $\omega \in \Omega$. Consequently, a compact cocycle is quasi-compact if and only if $\lambda^*\in \mathbb R$.

Then, we are ready to state the version of MET established in~\cite{Gonzalez-Tokman-Quas-2014}.

{\bf Multiplicative Ergodic Theorem (\cite{Gonzalez-Tokman-Quas-2014})}.
{\it Suppose that $\Phi$ is a measurable quasi-compact linear cocycle over $\theta$ whose generator $A$ satisfies (\ref{intmet}).
Then one of the following alternatives holds:
\\
{\bf (M-i)} there is a finite sequence of numbers
$
\lambda_1> \lambda_2 >\ldots >\lambda_k> \kappa^{*}
$
and a decomposition
$
X=V(\omega)\oplus \bigoplus_{i=1}^{k}E_i(\omega)
$
for $\mathbb P$-a.e. $\omega \in \Omega$
such that
\[
A(\omega)E_i(\omega)=E_i(\theta(\omega)), \forall i=1, \ldots, k, \   \ A(\omega)V(\omega)\subset V(\theta(\omega))
\]
and
$
\lim_{n\to \infty}\frac 1 n \ln | \Phi(n,\omega)v|=\lambda_i
$
for $v\in E_i(\omega)\setminus \{0\}$ and $i\in \{1, \ldots, k\}$. Moreover, $\lim_{n\to \infty}\frac 1 n \ln | \Phi(n,\omega)v| \le \kappa^*$ for $\mathbb P$-a.e. $\omega \in \Omega$ and $v\in V(\omega)$.
Finally, $E_i(\omega)$ is a finite-dimensional subspace of $X$ for $i\in \{1, \ldots, k\}$;
\\
{\bf (M-ii)} there is an infinite sequence of numbers
$
\lambda_1 >\lambda_2 >\ldots >\lambda_k >\ldots
$
satisfying $\lim_{k\to \infty}\lambda_k=\kappa^{*}$ and
 a decomposition
$
X=V(\omega)\oplus \bigoplus_{i=1}^{\infty}E_i(\omega)
$
for $\mathbb P$-a.e. $\omega \in \Omega$
such that
\[
A(\omega)E_i(\omega)=E_i(\theta(\omega)), \forall i\in \N, \  A(\omega)V(\omega)\subset V(\theta(\omega))
\]
and
$
\lim_{n\to \infty}\frac 1 n \ln | \Phi(n,\omega)v| =\lambda_i
$
for $v\in E_i(\omega)\setminus \{0\}$ and $i\in \N$. Moreover, $\lim_{n\to \infty}\frac 1 n \ln | \Phi(n,\omega)v| \le \kappa^*$ for $\mathbb P$-a.e. $\omega \in \Omega$ and $v\in V(\omega)$.
Finally, $E_i(\omega)$ is a finite-dimensional subspace of $X$ for $i\in \N$.
}

\begin{remark}
The numbers $\lambda_i$ in the above statement
are called \emph{Lyapunov exponents} of the cocycle $\Phi$ with respect to $\mathbb P$.
The above is version of MET for infinite-dimensional random dynamical systems.
MET was established earlier in (\cite{Oseledets,Arnold}) for the finite-dimensional case.
Before \cite{Gonzalez-Tokman-Quas-2014}, MET for the infinite-dimensional case
was given in~\cite{Lian-Lu} with an additional assumption that the cocycle is injective.
\end{remark}

In the context of MET,
if all the Lyapunov exponents of $\Phi$ with respect to $\mathbb P$ are nonzero,
then in both alternatives there is a unique $k^{*}\in \mathbb{Z}_{+}$ ($k^{*}\leq k$ in the Alternative  {\bf (M-i)})
such that
$\lambda_{i}>0$ for all $i\leq k^{*}$ and $\lambda_{i}<0$ for all $i> k^{*}$
and
\[X=\widetilde{V}(\omega)\oplus \bigoplus_{i=1}^{k^{*}}E_i(\omega),\]
where \[\widetilde{V}(\omega)=
\begin{cases}
V(\omega)\oplus \bigoplus_{i=k^{*}+1}^{k}E_i(\omega)
 & \text{if Alternative {\bf (M-i)} holds;}\\
V(\omega)\oplus \bigoplus_{i=k^{*}+1}^{\infty}E_i(\omega) & \text{if Alternative {\bf (M-ii)}  holds.}
 \end{cases}
\]

\begin{proposition}(\cite[Proposition 3.2.]{Backes-Dragicevic-2019})
\label{TED}
Suppose that $\Phi$ is a measurable quasi-compact linear cocycle over $\theta$
whose generator $A$ satisfies~\eqref{intmet}.
Moreover, assume that all Lyapunov exponents of $\Phi$ with respect to $\mathbb P$ are nonzero.
Then, there exist $\alpha>0$,
 a tempered (with respect to $\theta$)
 random variable $K\colon \Omega \to (0, +\infty)$ and a $\theta$-invariant  $\widetilde{\Omega}\subset \Omega$ measurable set  of full measure such that
 \begin{align*}
 \|\Phi(n,\omega)\Pi^{s}(\omega)\|& \leq
 K(\omega)e^{-\alpha n},
 & \forall~ \omega\in \widetilde{\Omega},n\in \mathbb{Z}_{+},
\\
 \|\Phi(n,\omega)\Pi^{u}(\omega)\|& \leq
 K(\omega)e^{\alpha n},
 & \forall~ \omega\in\widetilde{\Omega},n\in \mathbb{Z}_{-},
\end{align*}
where
$\Phi(n,\omega):=(\Phi(\theta^{n}\omega, -n)|_{\mathcal{R}(\Pi^{u}(\theta^{n}\omega))})^{-1}$
for $\omega \in \widetilde{\Omega}$ and $n\in \mathbb{Z}_{-}$,
$\Pi^u(\omega):=\Id-\Pi^s(\omega)$, and  $\Pi^{s}(\omega)$ is the projection on $X$ such that
  \[\mathcal{R}(\Pi^{s}(\omega))=\widetilde{V}(\omega) \quad \text{and} \quad
  \mathcal{N}(\Pi^{s}(\omega))=\bigoplus_{i=1}^{k^{*}}E_i(\omega).\]
\end{proposition}


Now we are ready to give our second  converse of Theorem \ref{theorem-texd-admi}. When $\beta \equiv 0$, we  will write $\ell_{1/K(\omega)}^\infty(\Z)$ instead of $\ell_{1/K(\omega), \beta (\omega)}^\infty(\Z)$ and
$\ell_1^\infty(\Z)$ instead of $\ell_{1, \beta(\omega)}^\infty(\Z)$.  We adapt the same convention for $l^\infty_{1/K, \beta}(\widetilde{\Omega}\times \Z)$ and $l^\infty_{1, \beta}(\widetilde{\Omega}\times \Z)$.

\begin{theorem}\label{t2}
Suppose that $\Phi$ is a measurable quasi-compact  linear cocycle
over $\theta$ whose generator $A$ satisfies~\eqref{intmet}.
Furthermore, assume that there exist a $\theta$-invariant  Borel set $\widetilde{\Omega}\subset \Omega$ of full measure and
 random variables
 $\zeta,\eta: \Omega \to (0, \infty)$ such that
 the pair $(\ell_{1/\zeta}^{\infty}(\widetilde{\Omega}\times \mathbb{Z}),\ell_{1}^{\infty}(\widetilde{\Omega} \times \mathbb{Z}))$
 is measurably properly admissible for $\Phi$ and  that for any fixed $\omega\in \widetilde{\Omega}$
and $f\in \ell_{1/\zeta(\omega)}^{\infty}(\mathbb{Z})$
\begin{eqnarray}
\label{adm3}
\|x\|_{\ell_{1}^{\infty}(\mathbb{Z})}\leq \eta(\omega)\|f\|_{\ell_{1/\zeta(\omega)}^{\infty}(\mathbb{Z})},
\end{eqnarray}
 where $x\in \ell_{1}^{\infty}(\mathbb{Z})$ is the unique solution of the equation
\begin{equation}\label{adm}
x(n+1)=A(\theta^{n}\omega)x(n)+f(n+1).
\end{equation}
 Then, $\Phi$ has a tempered exponential dichotomy.
 \end{theorem}

\begin{remark}
Before providing the proof, we would like to observe that, in comparison with Theorem~\ref{theorem-admi-texd},
we require only one admissibility condition.
Furthermore,
we do not require $\zeta$ to be tempered or  $\eta$ to be  $\theta$-invariant.
\end{remark}

{\bf Proof of Theorem~\ref{t2}.}
By arguing as in the proof of Theorem~\ref{theorem-admi-texd} (with $\beta \equiv 0$), for each $\omega \in \widetilde{\Omega}$ we can construct a splitting
\begin{equation}
\label{X=S+U-dirsum}
X=S(\omega)\oplus U(\omega)
\end{equation}
such that
\begin{equation}
\label{est-pro-Ps-2}
\lVert \Pi(\omega)\rVert \le \eta(\omega)\zeta(\omega),
\end{equation}
where $\Pi(\omega)\colon X\to S(\omega)$ is the projection associated to~\eqref{X=S+U-dirsum}. In addition,
 $A(\omega)\rvert_{U(\omega)} \colon U(\omega) \to U(\theta\omega)$ is an isomorphism. We recall from the proof of Theorem~\ref{theorem-admi-texd} that
\[
S(\omega)=\bigg \{ v\in X:\sup_{n\ge 0} | \Phi(n, \omega)v|<+\infty \bigg \}.
\]
Moreover, $U(\omega)$ consists of all $v\in X$ with the property that there exists a sequence $(x(n))_{n\le 0}\subset X$ such that $x(0)=v$, $x(n)=A(\theta^{n-1}\omega)x(n-1)$ for $n\le 0$ and $\sup_{n\le 0}|x(n)|<+\infty$.
Finally, recall that~\eqref{est-pro-Ps-2} follows from~\eqref{G}.

By Luzin's theorem (see Lemma {\bf A1} in the Appendix or ~\cite[Theorem 7.10, p.217]{Folland}),
there exists a compact $F\subset \widetilde{\Omega}$ with $\mathbb P (F)>0$
such that $\zeta, \eta$, $\zeta\circ \theta$ and $\omega \mapsto \lVert A(\omega)\rVert$ are continuous on $F$.  Hence, there exists $L>0$ such that
\begin{equation}\label{b}
\sup_{\omega \in F} \max \{ \zeta(\omega), \eta(\omega)\} \le L
\end{equation}
and
\begin{equation}\label{bnew}
\sup_{\omega \in F} \max \{\zeta(\theta \omega), \lVert A(\omega)\rVert \}\le L.
\end{equation}
Applying Poincare's recurrence theorem (see~Lemma {\bf A2} or \cite[Theorem 4.1.19, p.142]{Katok-Hasselblatt-1995})
to both $\theta$ and $\theta^{-1}$,
we get
$\mathbb P(F')=\mathbb P(F)$, where
\[
F'=\{ \omega \in F: \text{$\theta^n \omega \in F$ for infinitely many $n>0$ and $n<0$} \}.
\]
Define $\bar{\theta} \colon F' \to F'$ as
\[
\bar{\theta} (\omega)=\theta^{\tau_1(\omega)} (\omega) \quad \text{for $\omega \in F'$,}
\]
where
\[
\tau_1(\omega)=\min \{n\ge 1: \theta^n \omega\in F'\}.
\]
Furthermore, for $\omega \in F'$ let
\[
\bar A(\omega):=\Phi(\tau_1(\omega),\omega) \quad \text{and} \quad B(\omega)=\bar A(\omega)\Pi(\omega).
\]
Finally, for $\omega \in F'$ and $n\in \Z_{+}$, let
\[
\bar {\Phi}(n,\omega):=\begin{cases}
\Phi(\tau_n (\omega),\omega) & \text{if $n\ge 1$;}\\
\Id &\text{if $n=0$,}
\end{cases}
\]
where
\[
\tau_n (\omega)=\sum_{j=0}^{n-1}\tau_1 (\bar{\theta}^j(\omega)).
\]
Note that
\[
\bar{\theta}^n (\omega)=\begin{cases}
\theta^{\tau_n (\omega)}(\omega) & \text{for $n\ge 1$;}\\
\omega & \text{for $n=0$,}
\end{cases}
\]
and
\[
\bar{\Phi}(n+m,\omega)=\bar{\Phi}(\bar{\theta}^n (\omega), m)\bar{\Phi}(n,\omega), \quad \text{for $\omega \in F'$ and $m, n\in \Z_{+}$.}
\]

\begin{lemma}\label{LL}
There exists a constant $h_{1}>0$ such that
for arbitrarily given $\omega \in F'$ and $(s(n))_{n\ge 0} \subset X$ satisfying $\sup_{n\ge 0}| s(n)|<\infty$,
there exists $(\tilde x(n))_{n\ge 0}\subset X$ with $\tilde x(0)=s(0)$
such that
\begin{eqnarray}
&&\tilde x(n+1)=B(\bar{\theta}^n (\omega))\tilde x(n)+s(n+1), \quad \text{for $n\in \Z_{+}$},
\label{adm4}
\\
&&\sup_{n\ge 0}| \tilde x(n)| \le h_{1}\sup_{n\ge 0}| s(n)|.
\label{Ne}
\end{eqnarray}
\end{lemma}

\begin{proof}[{\bf Proof}]
For any fixed $\omega \in F'$ and $(s(n))_{n\ge 0} \subset X$ with $\sup_{n\ge 0}|s(n)|<\infty$,
we define $f=(f(n))_{n\in \Z}\subset X$ as
\[
f(n)=\begin{cases}
s(0) & \text{for $n=0$;}\\
s(k) & \text{if $n=\tau_k(\omega)$;}\\
0 & \text{otherwise.}
\end{cases}
\]
By~\eqref{b}, we have that
\begin{equation}\label{322}
\lVert  f\rVert_{\ell_{1/\zeta(\omega)}^{\infty}(\mathbb{Z})}=\sup_{n\ge 0}\big{(}\zeta(\bar{\theta}^n (\omega))| s(n)|)\le L\sup_{n\ge 0}| s(n)|.
\end{equation}
Hence, $f\in \ell_{1/\zeta(\omega)}^{\infty}(\mathbb{Z})$. Thus, there exists a unique $x=(x(n))_{n\in \Z}\in \ell_{1}^{\infty}(\mathbb{Z})$ such that~\eqref{adm} holds.  In particular, we have that
\[
x(0)-A(\theta^{-1}\omega)x(-1)=s(0)
\]
and
\[
x(n)=A(\theta^{n-1}\omega)x(n-1), \quad  \text{for $n\le -1$.}
\]
Hence, $x(0)-s(0)\in U(\omega)$ and therefore
\begin{equation}\label{1816}
\Pi(\omega)x(0)=\Pi(\omega)s(0).
\end{equation}
We now define $(z(n))_{n\ge 0}\subset X$ as
\[
z(n)=\begin{cases}
x(0) & \text{for $n=0$;} \\
x(\tau_n(\omega)) & \text{for $n>0$.}
\end{cases}
\]
It follows from~\eqref{adm} that
\begin{equation}\label{308}
z(n+1)-\bar A(\bar \theta^n (\omega))z(n)=s(n+1), \quad \text{for $n\in \Z_{+}$.}
\end{equation}
Let
\[
z'(n):=\Pi(\bar \theta^n (\omega))z(n), \quad n\in \Z_{+}.
\]
By~\eqref{308}, we have that
\begin{equation}\label{bc}
z'(n+1)-\bar A(\bar \theta^n (\omega)) z'(n)=\Pi (\bar \theta^{n+1}(\omega))s(n+1), \quad \text{for $n\in \Z_{+}$.}
\end{equation}
Furthermore (see~\eqref{1816}),
\begin{equation}\label{z0}
z'(0)=\Pi(\omega)z(0)=\Pi(\omega)x(0)=\Pi(\omega)s(0).
\end{equation}
Moreover, it follows from~\eqref{adm3},  (\ref{est-pro-Ps-2}),  \eqref{b} and~\eqref{322} that
\[
\sup_{n\ge 0}| z'(n)| \le L^2 \sup_{n\ge 0}| z(n)| \le L^2\lVert x\rVert_{\ell_{1}^{\infty}(\mathbb{Z})}
\le L^3\lVert  f\rVert_{\ell_{1/\zeta(\omega)}^{\infty}(\mathbb{Z})} \le L^4 \sup_{n\ge 0}| s(n)|,
\]
which implies that
\begin{equation}\label{3:29}
\sup_{n\ge 0}| z'(n)| \le L^4 \sup_{n\ge 0}| s(n)|.
\end{equation}
Let
\[
\tilde x(n):=z'(n)+(\Id-\Pi(\bar \theta^n (\omega)))s(n), \quad n\in \Z_{+}.
\]
Observe that~\eqref{z0} implies that
\[
\tilde x(0)=\Pi(\omega)s(0)+(\Id-\Pi(\omega))s(0)=s(0).
\]
In addition, \eqref{bc} and~\eqref{z0} give that
\[
z'(n)=\sum_{k=0}^n \bar{\Phi}(n-k,\bar{\theta}^k (\omega))\Pi (\bar{\theta}^k (\omega))s(k), \quad \text{for $n\in \Z_{+}$.}
\]
Hence, we have that
\begin{eqnarray}
\label{xB-tlx=s}
&&\tilde x(n+1)-B(\bar{\theta}^n (\omega)) \tilde x(n)
\\
&=&\sum_{k=0}^{n+1} \bar{\Phi}(n+1-k,\bar{\theta}^k (\omega))\Pi (\bar{\theta}^k (\omega))s(k)\nonumber \\
&&\phantom{=}-\sum_{k=0}^n \bar{\Phi}(n+1-k,\bar{\theta}^k (\omega))\Pi (\bar{\theta}^k (\omega))s(k)\nonumber \\
&&\phantom{=}+(\Id-\Pi(\bar \theta^{n+1} (\omega)))s(n+1)\nonumber\\
&=&\Pi(\bar \theta^{n+1} (\omega))s(n+1)+(\Id-\Pi(\bar \theta^{n+1} (\omega)))s(n+1)\nonumber \\
&=&s(n+1),
\nonumber
\end{eqnarray}
for each $n\in \Z_{+}$. Consequently, \eqref{adm4} holds.  Finally, (\ref{est-pro-Ps-2}), \eqref{b} and~\eqref{3:29} imply that~\eqref{Ne} holds with
\[
h_{1}=1+L^2+L^4>0.
\]
The proof of the lemma is completed.
\end{proof}

\begin{lemma}\label{ww}
There exist $Q>0$ and $\nu>0$ such that
\[
| \bar{\Phi}(n,\omega)v| \le Qe^{-\nu n}|v|, \quad \text{for $\omega \in F'$, $v\in S(\omega)$ and $n>0$.}
\]
\end{lemma}

\begin{proof}[{\bf Proof}]
For any fixed $\omega\in F'$ and  $v\in S(\omega)$,
define $(s(n))_{n\ge 0}\subset X$ by
\[
s(n)=\begin{cases}
v & \text{for $n=0$;}\\
0& \text{for $n\neq 0$.}
\end{cases}
\]
Obviously, $\sup_{n\ge 0}|s(n)|=| v|$. Let $(\tilde x(n))_{n\ge 0}\subset X$ be the associated sequence given by Lemma~\ref{LL}. Observe that
\[
\tilde x(n)=\Phi(\tau_n (\omega),\omega)v, \quad n\in \N.
\]
Hence, \eqref{Ne} implies that
\[
| \Phi(\tau_n (\omega),\omega)v| \le h_{1} | v|, \quad \text{for $n\in \N$.}
\]
Since $v\in S(\omega)$ was arbitrary, we conclude that
\begin{equation}\label{5:08}
| \bar{\Phi}(n,\omega)v| \le h_{1}|v|, \quad \text{for $\omega \in F'$, $v\in S(\omega)$ and $n\in \N$.}
\end{equation}

Let $n_0\in \N$ satisfy
\[
\frac{h_{1}^2}{n_0+1}\le \frac 1 e.
\]
For any fixed $\omega \in F'$ and $v\in S(\omega)$, we define $(\tilde{s}(n))_{n\ge 0}\subset X$ as
\[
\tilde{s}(n)=\begin{cases}
\bar{\Phi}(n,\omega)v & \text{for $0\le n\le n_0$;}\\
0 & \text{for $n>n_0$.}
\end{cases}
\]
By~\eqref{5:08}, we see that $\sup_{n\ge 0}| \tilde{s}(n)| \le h_{1}| v|$.
Let $(\tilde x(n))_{n\ge 0}\subset X$ be the associated sequence given by Lemma~\ref{LL}. Note that
\[
\tilde x(n_0)=\sum_{k=0}^{n_0}\bar{\Phi}(n_0-k,\bar{\theta}^k (\omega))\bar{\Phi}(k,\omega)v=(n_0+1)\bar{\Phi}(n_0,\omega)v.
\]
Hence, \eqref{Ne} implies that
\[
(n_0+1)| \bar{\Phi}(n_0,\omega)v| \le h_{1}\sup_{n\ge 0}| \tilde{s}(n)| \le h_{1}^2| v|,
\]
which implies that
\[
|\bar{\Phi}(n_0,\omega)v|  \le \frac{1}{e}| v|.
\]
Thus,
\begin{equation}\label{6:11}
| \bar{\Phi}(n_0,\omega)v|  \le \frac{1}{e}|v|, \quad \text{for $\omega \in F'$ and $v\in S(\omega)$.}
\end{equation}
It follows from \eqref{5:08} and~\eqref{6:11} that the conclusion of the lemma holds with
$Q=h_{1}e$ and $\nu=\frac{1}{n_0}$.
\end{proof}

\begin{lemma}\label{LE}
For $\mathbb P$-a.e. $\omega \in F'$ and every $v\in S(\omega)$,
\[
\lim_{n\to \infty}\frac 1 n \log | \Phi(n,\omega)v|<0.
\]
\end{lemma}

\begin{proof}[{\bf Proof}]
The existence of the limit (for $\mathbb P$-a.e. $\omega \in F'$) follows from MET.
Next, we have that
\begin{eqnarray}
\label{Phi=PF-1}
\lim_{n\to \infty}\frac 1 n \log |\Phi(n,\omega)v|
&=&\lim_{n\to \infty}\frac{1}{\tau_n(\omega)} \log | \Phi(\tau_n(\omega), \omega)v|
\\
&=&\lim_{n\to \infty}\frac{1}{\tau_n(\omega)}\log | \bar{\Phi}(n,\omega)v|
\nonumber\\
&=&\lim_{n\to \infty}\frac{n}{\tau_n(\omega)}\cdot \lim_{n\to \infty}\frac 1 n \log | \bar{\Phi}(n,\omega)v|
\nonumber
\end{eqnarray}
for $\mathbb P$-a.e. $\omega \in F'$ and every $v\in S(\omega)$.
We claim that
\begin{equation}\label{pp}
\lim_{n\to \infty}\frac{\tau_n(\omega)}{n}=\frac{1}{\mathbb P (F)}, \quad \text{for $\mathbb P$-a.e. $\omega \in F'$.}
\end{equation}
In fact, one can check (see~ \cite[Theorem 1.7]{Sarig}) that the measure-preserving system $(F', \mathcal F_{F'}, \mathbb P_{F'}, \bar{\theta})$ is ergodic, where
$\mathcal F_{F'}=\{A\cap F': A\in \mathcal F\}$ and $\mathbb P_{F'}$ is given by
\[
\mathbb P_{F'}(B)=\frac{\mathbb P(B)}{\mathbb P(F')}, \quad B\in \mathcal F_{F'}.
\]
It follows from Kac's lemma (see Lemma~{\bf A4} or \cite[Theorem 2', p.1006]{Kac-1947})
that $\tau_1\in \mathcal{L}^{1}(F',\mathbb P_{F'})$.
By Birkhoff's ergodic theorem (see~Lemma~{\bf A3} or \cite[Theorem 4.1.2, p.136]{Katok-Hasselblatt-1995})
and Kac's lemma, we get that
\begin{eqnarray*}
\lim_{n\to \infty}\frac{\tau_n(\omega)}{n}
=\lim_{n\to \infty}\frac{1}{n}\sum_{j=0}^{n-1}\tau_1 (\bar{\theta}^j(\omega))
=\int_{F'}\tau_1(\omega)d\mathbb P_{F'}
=\frac{1}{\mathbb P(F')},
\end{eqnarray*}
for $\mathbb P$-a.e. $\omega \in F'$.
Since $\mathbb P(F')=\mathbb P(F)$, we have that~\eqref{pp} holds.
It follows from (\ref{Phi=PF-1}) and (\ref{pp}) that
\begin{eqnarray*}
\lim_{n\to \infty}\frac 1 n \log |\Phi(n,\omega)v|
=\mathbb P(F)\cdot \lim_{n\to \infty}\frac 1 n \log |\bar{\Phi}(n,\omega)v|,
\end{eqnarray*}
for $\mathbb P$-a.e. $\omega \in F'$ and every $v\in S(\omega)$.
Now the conclusion of the lemma follows directly from Lemma~\ref{ww}.
\end{proof}

For $\omega \in F'$, let
\[
\tau_1'(\omega):=\min \{n\ge 1: \theta^{-n}\omega\in F'\}.
\]
Furthermore, we define $\tilde \theta \colon F' \to F'$ as
\[
\tilde \theta(\omega)=\theta^{-\tau_1'(\omega)}(\omega), \quad \omega \in F'.
\]
Observe that $\tilde \theta=\bar{\theta}^{-1}$. In addition, for $n\ge 1$ and $\omega \in F'$, let
\[
\tau_n'(\omega):=\sum_{j=0}^{n-1}\tau_1' (\tilde \theta^j (\omega)).
\]
We then have
\[
\tilde \theta^{n} (\omega)=\theta^{-\tau_n'(\omega)}(\omega), \quad \text{for $\omega \in F'$ and $n\ge 1$.}
\]
Finally, let
\[
C(\omega):=\bar{A}(\tilde \theta(\omega))^{-1}(\Id-\Pi(\omega)), \quad \omega \in F'.
\]

\begin{lemma}
There exists a constant $h_{2}>0$ such that
for arbitrarily given $\omega \in F'$ and $(s(n))_{n\ge 0} \subset X$ satisfying $\sup_{n\ge 0}| s(n)|<\infty$
there exists $(\tilde x(n))_{n\ge 0}\subset X$ with $\tilde x(0)=s(0)$ such that
\begin{eqnarray}
\label{adm5}
&&\tilde x(n+1)=C(\tilde{\theta}^n (\omega))\tilde x(n)+s(n+1) \quad \text{for $n\in \Z_{+}$},
\\
\label{Ne2}
&&\sup_{n\ge 0}|\tilde x(n)| \le h_{2}\sup_{n\ge 0}| s(n)|.
\end{eqnarray}
\label{lm14}
\end{lemma}

\begin{proof}[{\bf Proof}]
The proof of Lemma \ref{lm14} is similar to the proof of of Lemma \ref{LL}.
Since $A(\omega)$ is not invertible on the entire $X$,
for the sake of completeness, we provide a short proof.
For any fixed $\omega \in F'$ and $(s(n))_{n\ge 0} \subset X$ with $\sup_{n\ge 0}|s(n)|<\infty$,
we define $f=(f(n))_{n\in \Z}\subset X$ as
\[
f(n)=\begin{cases}
-A(\omega) s(0) & \quad \text{for $n=1$;} \\
-A(\tilde \theta^k (\omega))s(k) & \quad \text{if $n=-\tau_k'(\omega)+1$ for some $k\ge 1$;}\\
0 & \quad \text{otherwise.}
\end{cases}
\]
By~\eqref{bnew}, we have that
\begin{equation}\label{11a}
\lVert f\rVert_{\ell_{1/\zeta(\omega)}^{\infty}(\mathbb{Z})}=\sup_{n\ge 0} \big  (\zeta(\theta (\tilde \theta^n (\omega)))| A(\tilde \theta^n (\omega))s(n)| \big)\le L^2 \sup_{n\ge 0}| s(n)|.
\end{equation}
Hence, $f\in \ell_{1/\zeta(\omega)}^{\infty}(\mathbb{Z})$. Thus, there exists a unique $x=(x(n))_{n\in \Z}\in \ell_{1}^{\infty}(\mathbb{Z})$ such that~\eqref{adm} holds. Thus,
\[
x(1)=A(\omega)(x(0)-s(0)),
\
x(n)=A(\theta^{n-1}(\omega))x(n-1), \quad \text{for $n\ge 2$.}
\]
It follows that $x(0)-s(0)\in S(\omega)$ and hence
\begin{equation}\label{ee}
(\Id-\Pi(\omega))x(0)=(\Id-\Pi(\omega))s(0).
\end{equation}
We now define $(z(n))_{n\le 0}\subset X$ by
\[
z(n)=\begin{cases}
x(0) & \text{for $n=0$;}\\
x(-\tau_{-n}'(\omega)) & \text{for $n<0$.}
\end{cases}
\]
Then, from~\eqref{adm} we obtain that
\begin{equation}\label{a1q}
z(n) =\bar{A}(\tilde \theta^{-n+1}(\omega))(z(n-1)-s(-n+1)), \quad \text{for $\omega \in F'$ and $n\le 0$.}
\end{equation}
Let
\[
z'(n):=(\Id-\Pi(\tilde \theta^{-n}(\omega)))z(n), \quad n\le 0.
\]
Then, \eqref{a1q} implies that
\begin{equation}\label{ab}
z'(n)=\bar{A}(\tilde \theta^{-n+1}(\omega))[z'(n-1)-(\Id-\Pi(\tilde \theta^{-n+1}(\omega)))s(-n+1)], \quad n\le 0.
\end{equation}
In addition, \eqref{ee} implies that
\begin{equation}\label{ab1}
z'(0)=(\Id-\Pi(\omega))s(0).
\end{equation}
Similarly to (\ref{3:29}), one can check that
\begin{equation}\label{uu}
\sup_{n\le 0}| z'(n)| \le L^3(1+L^2)\sup_{n\ge 0}| s(n)|.
\end{equation}
On the other hand, \eqref{ab} and~\eqref{ab1} imply that
\begin{equation*}\label{33}
z'(-n)=\sum_{k=0}^n \bar{\Phi}(n-k,\tilde \theta^n (\omega))^{-1}(\Id-\Pi(\tilde \theta^k (\omega)))s(k), \quad \text{for $n\in \Z_{+}$.}
\end{equation*}
Finally, let
\[
\tilde x(n):=z'(-n)+\Pi(\tilde \theta^n (\omega))s(n), \quad n\in \N.
\]
Similarly to (\ref{xB-tlx=s}), one can check that~\eqref{adm5} holds.
Moreover, by (\ref{est-pro-Ps-2}), \eqref{b} and~\eqref{uu}, we conclude that~\eqref{Ne2} holds with the constant
\[
h_{2}:= L^3(1+L^2)+L^2>0.
\]
The proof of the lemma is completed.
\end{proof}
The proof of the following lemma is analogous to the proof of Lemma~\ref{ww}.
\begin{lemma}\label{ww1}
There exist $Q', \nu'>0$ such that
\[
| \bar{\Phi}(n,\tilde \theta^n (\omega))^{-1}v| \le Q'e^{-\nu' n}| v|, \quad \text{for $\omega \in F'$, $v\in U(\omega)$ and $n\in \N$.}
\]
\end{lemma}
It follows from Lemma~\ref{ww1} that
\[
\frac{1}{Q'}e^{\nu' n}| v| \le | \bar{\Phi}(n,\omega)v|, \quad \text{for $\omega \in F'$, $v\in U(\omega)$ and $n\in \N$.}
\]
Now, arguing as in the proof of Lemma~\ref{LE}, one can obtain the following lemma.

\begin{lemma}\label{LE2}
$
\lim_{n\to \infty}\frac 1 n \log | \Phi(n,\omega)v|>0
$
for $\mathbb P$-a.e. $\omega \in F'$ and every $v\in U(\omega)\setminus \{0\}$.
\end{lemma}

As a direct consequence of (\ref{X=S+U-dirsum}), Lemmas \ref{LE} and~\ref{LE2}, we conclude that for  $\mathbb P$-a.e. $\omega \in F'$
 and every $v\in X$, we have that
\begin{equation}\label{Le}
\lim_{n\to \infty}\frac 1 n \log | \Phi(n,\omega)v| \neq 0.
\end{equation}
We now claim that all Lyapunov exponents of $\Phi$ with respect to $\mathbb P$ are nonzero.
Assume the opposite, i.e., there exists $i\in \N$ such that $\lambda_i=0$. Let $\Omega'\subset \Omega$ be the measurable set on which all
conclusions of MET are true. Note that
$\mathbb P(\Omega')=1$. Let $F''\subset F'$ be a measurable set, $\mathbb P(F'')=\mathbb P(F')$ such that~\eqref{Le} holds for each $\omega \in F''$ and $v\in X$.
Since $\mathbb P(F'')>0$, we have that $F''\cap \Omega'\neq \emptyset$.
Since $\lambda_i=0$,
for any fixed $\omega_0\in F''\cap \Omega'$ and for every $v\in E_i(\omega_0)\setminus \{0\}$ we have
\[
\lim_{n\to \infty}\frac 1 n \log | \Phi(n,\omega_0)v|=0,
\]
which contradicts~\eqref{Le}.
The conclusion of the theorem now follows directly from Proposition ~\ref{TED}.
 \hfill$\square$


\section{Applications
}\label{T4}

In this section,
we use Theorem \ref{theorem-admi-texd}
to give the roughness of tempered exponential dichotomies for parametric random systems.
Moreover, we use the approach in the proof of Theorem \ref{theorem-admi-texd}
to obtain a corresponding result for the deterministic difference equation (\ref{lnde-dis}).

First, we give the roughness of tempered exponential dichotomies for parametric random systems
and formulate conditions for a H\"older continuous dependence of the associated projections on the parameter.

\begin{cor}\label{RTED}
Suppose that $(\Xi, |\cdot |)$ is a normed space, and
 $\Phi$ and $\Psi_{\xi}$ for  $\xi\in \Xi$ are measurable linear cocycles over $\theta$
  with the generators $A$ and $B_{\xi}$ respectively.
In addition, assume
\begin{itemize}
\item[(A1)]
 $\Phi$ has a tempered exponential dichotomy on $\widetilde \Omega \subset \Omega$
 with the exponent $\alpha \colon \Omega \to (0, +\infty)$ and
 bound $K\colon \Omega \to (0, +\infty)$;

\item[(A2)] for $\omega \in \widetilde{\Omega}$,
\begin{equation}\label{rob}
\|A(\omega)-B_{\xi}(\omega)\| \le \frac{\rho(\omega)}{K(\theta \omega)},
\ \forall \ \xi\in \Xi,
\end{equation}
where $\rho\colon \Omega \to (0, +\infty)$ is a $\theta$-invariant random variable such that
\begin{equation}\label{rho}
\rho(\omega)< \frac{1-e^{-\alpha (\omega)}}{1+e^{-\alpha (\omega)}}, \quad \omega \in \Omega;
\end{equation}

\item[(A3)] for $\omega \in \widetilde{\Omega}$,
\begin{equation}\label{Hold-cond}
\|B_{\xi_{1}}(\omega)-B_{\xi_{2}}(\omega)\| \le \frac{\upsilon(\omega)|\xi_{1}-\xi_{2}|^{\sigma}}{K(\theta \omega)},
\ \forall \ \xi_{1},\xi_{2}\in \Xi,
\end{equation}
where $\upsilon\colon \Omega \to (0, +\infty)$ is a $\theta$-invariant random variable
and $\sigma \in (0,1]$ is a constant.
\end{itemize}
Then, $\Psi_{\xi}$ has a tempered exponential dichotomy on $\widetilde \Omega$ for any $\xi\in \Xi$
with the same exponent $\widetilde{\alpha} \colon \Omega \to (0, +\infty)$ and
 bound $\widetilde{K}\colon \Omega \to (0, +\infty)$. Moreover,
 the corresponding projections $\widetilde{\Pi}_{\xi}(\omega)$ are H\"older continuous in $\xi$
with the H\"older exponent $\sigma$.
\end{cor}

\begin{proof}
 In order to check the conditions in Theorem \ref{theorem-admi-texd},
we need the following lemma, the proof of which
is omitted since it is analogous to the proof of~\cite[Lemma 1. p.4029]{ZhouLuZhang13JDE}.

\begin{lemma}
Suppose that $\Phi$ is a measurable  linear cocycle
over $\theta$ with the generator $A$
and has a tempered exponential dichotomy on $\widetilde{\Omega}$
with the exponent $\alpha \colon \Omega \to (0, +\infty)$ and bound $K\colon \Omega \to (0, +\infty)$.
 Let $C \colon \Omega \to \mathcal {L}(X)$ be strongly measurable such that  $\|C(\omega)\|\leq \delta(\omega) K(\theta \omega)^{-1}$
 for all $\omega \in  \widetilde{\Omega}$,
where $\delta: \Omega\rightarrow (0,+\infty)$ is a $\theta$-invariant random variable.
  Then
  for any $\theta$-invariant random variable $\beta \colon \Omega \to \R$ such that $\beta (\omega)\in (-\alpha (\omega), \alpha (\omega))$ for $\omega \in \Omega$
  and
  for each $\omega\in \widetilde{\Omega}$ and
  $f\in \ell^\infty_{1/K(\omega), \beta(\omega)}(J)$ (resp. $f\in \ell^\infty_{|1/K(\omega), \beta(\omega)|}(J)$),
 where $J=\mathbb{Z}_{+}$, $\mathbb{Z}_{-}$ or $\mathbb{Z}$,
  the equation
$$
x(n+1)=(A(\theta^{n}\omega)+C(\theta^{n}\omega))x(n)+f(n+1)
$$
  has a solution in
  $\ell^\infty_{1, \beta(\omega)}(J)$ (resp. $\ell^\infty_{|1, \beta(\omega)|}(J)$)
  if and only if there exists
  $x(\cdot,\omega)\in \ell^\infty_{1, \beta(\omega)}(J)$
  (resp. $x(\cdot,\omega)\in \ell^\infty_{|1, \beta(\omega)|}(J)$)
  satisfying one of the following equations
\begin{align}
\label{wt-sol-formula-Z+}
\begin{split}
x(n,\omega) &=\Phi(n,\omega)\Pi^{s}(\omega)\varsigma
\\
   & \ \      + \sum_{k=1}^{+\infty} G(n-k,\theta^{k}\omega)\bigg{\{}C(\theta^{k-1}\omega)x(k-1,\omega)+f(k)\bigg{\}},
    \ n\in \mathbb{Z}_{+},
\end{split}
\end{align}
for some $\varsigma\in X$ in the case $J=\mathbb{Z}_{+}$,
\begin{align}
\label{wt-sol-formula-Z-}
\begin{split}
x(n,\omega)&=\Phi(n,\omega)\Pi^{u}(\omega)\varsigma
\\
& \ \ \
 +\sum_{k=-\infty}^{0}G(n-k,\theta^{k}\omega)\bigg{\{}C(\theta^{k-1}\omega)x(k-1,\omega)+f(k)\bigg{\}},
 \ n\in \mathbb{Z}_{-},
\end{split}
\end{align}
for some $\varsigma \in X$ in the case $J=\mathbb{Z}_{-}$, and
\begin{eqnarray}
 x(n,\omega)
 = \sum_{k=-\infty}^{+\infty} G(n-k,\theta^{k}\omega)\bigg{\{}C(\theta^{k-1}\omega)x(k-1,\omega)+f(k)\bigg{\}},
 \  n\in \mathbb{Z},
\label{wt-sol-formula}
\end{eqnarray}
in the case $J=\mathbb{Z}$, where
$G(n,\omega)$ is the Green function corresponding to the tempered exponential dichotomy for $\Phi$ (as seen in \eqref{GGG-Temp}).
 \label{lm-wt-sol}
\end{lemma}

In order to apply Lemma \ref{lm-wt-sol},
we claim that there exists
 a positive $\theta$-invariant random variable $\beta$ such that
$\beta(\omega)\in (0,\alpha(\omega))$ for all $\omega\in \Omega$ and
\begin{equation}\label{6:07}
 \rho(\omega)e^{\beta(\omega)}
  \bigg{\{}\frac{1}{1-e^{-(\alpha(\omega)-\beta(\omega))}}
 +\frac{e^{-(\alpha(\omega)-\beta(\omega))}}{1-e^{-(\alpha(\omega)-\beta(\omega))}} \bigg{\}} <1.
\end{equation}
In fact, from the equality
\begin{eqnarray}
\label{contraction-constant}
 \rho(\omega)e^{\beta(\omega)}
  \bigg{\{}\frac{1}{1-e^{-(\alpha(\omega)-\beta(\omega))}}
 +\frac{e^{-(\alpha(\omega)-\beta(\omega))}}{1-e^{-(\alpha(\omega)-\beta(\omega))}} \bigg{\}}
  \!=\!\frac{1+\rho(\omega)
  \frac{1+e^{-\alpha(\omega)}}{1-e^{-\alpha(\omega)}}}{2},
\end{eqnarray}
one can find
\begin{eqnarray}
\label{est-beta-1}
 \beta(\omega)=\ln \bigg{\{}\frac{-\{2\rho(\omega)+
  e^{-\alpha(\omega)}\{1+\rho(\omega)
  \frac{1+e^{-\alpha(\omega)}}{1-e^{-\alpha(\omega)}}\}\}
  +\sqrt{\Delta(\omega)}}{4\rho(\omega)e^{-\alpha(\omega)}}\bigg{\}},
\end{eqnarray}
where
\begin{eqnarray*}
 &&\Delta(\omega):=\bigg{\{} 2\rho(\omega)+
  e^{-\alpha(\omega)} \bigg{\{}1+\rho(\omega)
  \frac{1+e^{-\alpha(\omega)}}{1-e^{-\alpha(\omega)}} \bigg{\}} \bigg{\}}^{2}
  \\
 &&\ \ \ \ \ \ \ \ \ \ \
  +8\rho(\omega)e^{-\alpha(\omega)}
\bigg{\{}1+\rho(\omega)
  \frac{1+e^{-\alpha(\omega)}}{1-e^{-\alpha(\omega)}} \bigg{\}}.
\end{eqnarray*}
Note that the function $\beta$ defined in (\ref{est-beta-1})
is a positive $\theta$-invariant random variable
because both $\alpha$ and $\rho$ are positive $\theta$-invariant random variables by the assumptions.
Moreover, by (\ref{rho}) and (\ref{contraction-constant}) one can check that
the random variable $\beta$ defined in (\ref{est-beta-1}) satisfies~\eqref{6:07}.
This proves the claim.

Let $\beta$ be a fixed positive $\theta$-invariant random variable satisfying (\ref{6:07})
and let $\lambda=\pm \beta$.
In order to prove the proper admissibility of the pair
$(\ell_{1/K,\lambda}^{\infty}(\widetilde \Omega \times \mathbb{Z}),\ell_{1,\lambda}^{\infty}(\widetilde \Omega \times \mathbb{Z}))$ for $\Psi_{\xi}$,
 by Lemma \ref{lm-wt-sol},
 it is sufficient to prove that for each
 $f\in \ell_{1/K(\omega),\lambda(\omega)}^{\infty}(\mathbb{Z})$ with $\omega\in \widetilde{\Omega}$
 there exists a unique sequence
 $x(\cdot,\omega)\in \ell_{1,\lambda(\omega)}^{\infty}(\mathbb{Z})$
 such that (\ref{wt-sol-formula}) holds
 with
 $C(\omega):=B_{\xi}(\omega)-A(\omega)$ depending on $\xi\in \Xi$ for each $\omega \in \Omega$.
 For convenience, we write $C$ as $C_\xi$ to indicate the dependence on $\xi$.

 For a given
  $f\in \ell_{1/K(\omega),\lambda(\omega)}^{\infty}(\mathbb{Z})$,
  consider a map
  $\Upsilon_{1,\lambda(\omega)}^f : \ell_{1,\lambda(\omega)}^{\infty}(\mathbb{Z})\rightarrow \ell_{1,\lambda(\omega)}^{\infty}(\mathbb{Z})$
  defined by
\begin{equation*}
(\Upsilon_{1,\lambda(\omega)}^{f} x)(n)
 :=
 \sum_{k=-\infty}^{+\infty} G(n-k,\theta^{k}\omega)\bigg{\{}C_{\xi}(\theta^{k-1}\omega)x(k-1)+f(k)\bigg{\}},
    \  n\in \mathbb{Z}.
\end{equation*}
Note from \eqref{rob} that
$(C_{\xi}(\theta^{n-1}\omega)x(n))_{n\in \mathbb{Z}}\in \ell_{1/K(\omega),\lambda(\omega)}^{\infty}(\mathbb{Z})$
for $x\in \ell_{1,\lambda(\omega)}^{\infty}(\mathbb{Z})$.
Hence, similarly to (\ref{de}),
by (\ref{green}) one can check that $\Upsilon_{1,\lambda(\omega)}^{f} x\in \ell_{1,\lambda(\omega)}^{\infty}(\mathbb{Z})$.
This implies that
 $\Upsilon_{1,\lambda(\omega)}^{f}(\ell_{1,\lambda(\omega)}^{\infty}(\mathbb{Z}))\subset \ell_{1,\lambda(\omega)}^{\infty}(\mathbb{Z})$.
 Moreover, for $x_{i}\in \ell_{1,\lambda(\omega)}^{\infty}(\mathbb{Z})$, $i=1,2$,
 by (\ref{green}) we get
\begin{eqnarray*}
 &&|(\Upsilon_{1,\lambda(\omega)}^{f} x_{1})(n)- (\Upsilon_{1,\lambda(\omega)}^{f} x_{2})(n)| \\
 &\leq&
 \sum_{k=-\infty}^{+\infty}K(\theta^{k}\omega)e^{-\alpha(\omega)|n-k|}
  \rho(\omega)K(\theta^{k}\omega)^{-1}
   |x_{1}(k-1)-x_{2}(k-1)|
 \\
  &\leq&
 \rho(\omega){\bigg\{}\sum_{k=-\infty}^{+\infty}e^{-\alpha(\omega)|n-k|}
   e^{\lambda(\omega)(k-1)}
  {\bigg\}}
  \|x_{1}-x_{2}\|_{\ell_{1,\lambda(\omega)}^{\infty}(\mathbb{Z})}
    \\
  &=&
  \rho(\omega)e^{\lambda(\omega)(n-1)}\widetilde{\gamma}_{\lambda}(\omega)
  \|x_{1}-x_{2}\|_{\ell_{1,\lambda(\omega)}^{\infty}(\mathbb{Z})},
\end{eqnarray*}
where $\widetilde{\gamma}_{\lambda}(\omega):=\frac{1}{1-e^{-(\alpha(\omega)+\lambda(\omega))}}
  +
  \frac{e^{-(\alpha(\omega)-\lambda(\omega))}}{1-e^{-(\alpha(\omega)-\lambda(\omega))}}$,
which implies that
\begin{eqnarray*}
 \|\Upsilon_{1,\lambda(\omega)}^{f} x_{1}-\Upsilon_{1,\lambda(\omega)}^{f} x_{2}\|_{\ell_{1,\lambda(\omega)}^{\infty}(\mathbb{Z})}
 \leq
  \rho(\omega)e^{-\lambda(\omega)}
  \widetilde{\gamma}_{\lambda}(\omega)
 \|x_{1}-x_{2}\|_{\ell_{1,\lambda(\omega)}^{\infty}(\mathbb{Z})}.
\end{eqnarray*}
By (\ref{6:07}),
\begin{eqnarray}
\label{est-gamma-1}
  &&\rho(\omega)e^{-\lambda(\omega)}
 \widetilde{\gamma}_{\lambda}(\omega)
 \\
 &\leq& \rho(\omega)e^{\beta(\omega)}
  \bigg{\{}\frac{1}{1-e^{-(\alpha(\omega)-\beta(\omega))}}
 +\frac{e^{-(\alpha(\omega)-\beta(\omega))}}{1-e^{-(\alpha(\omega)-\beta(\omega))}} \bigg{\}}<1.
 \nonumber
\end{eqnarray}
It follows
that $\Upsilon_{1,\lambda(\omega)}^{f}$
is a contraction on  the Banach space $\ell_{1,\lambda(\omega)}^{\infty}(\mathbb{Z})$.
This together with Lemma \ref{lm-wt-sol} implies that the pair
$(\ell_{1/K,\lambda}^{\infty}(\widetilde \Omega \times \mathbb{Z}),\ell_{1,\lambda}^{\infty}(\widetilde \Omega \times \mathbb{Z}))$
is properly admissible for $\Psi_{\xi}$.
 Moreover,
 for any fixed measurable map $f\colon \widetilde \Omega \times \Z \to X$ such that $(f(\omega, n))_{n\in \Z}\in \ell_{1/K(\omega),\beta(\omega)}^{\infty}(\mathbb{Z})$ for $\omega \in \widetilde \Omega$,
 let $x$ denote the solution of (\ref{madm}) with the property that
 $x(\omega,\cdot)$ is the fixed point of $\Upsilon_{1,\lambda(\omega)}^{f(\omega,\cdot)}$
 for all $\omega\in \widetilde{\Omega}$.
Let $x_{0}\equiv 0$ and $x_{i}(\omega,n):=(\Upsilon_{1,\lambda(\omega)}^{f(\omega,\cdot)}x_{i-1}(\omega,\cdot))(n)$
for $i\ge 1$.
Then
$x(\cdot,n)$ is the pointwise limit of the sequence of maps $(x_{i}(\cdot,n))_{i\in \mathbb{Z}_{+}}$ for any fixed $n\in \mathbb{Z}$,
because $\Upsilon_{1,\lambda(\omega)}^{f(\omega,\cdot)}$ is a contraction for any fixed $\omega\in \widetilde \Omega$. In addition,
it is clear that $\omega\mapsto x_{0}(\omega,n)$ is measurable for any fixed $n\in \mathbb{Z}$. Assume now that
for $i\ge 1$,  $\omega\mapsto x_{i-1}(\omega,n)$ is measurable for any fixed $n\in \mathbb{Z}$. Since
 \begin{equation*}
x_{i}(\omega,n)
 =
 \sum_{k=-\infty}^{+\infty} G(n-k,\theta^{k}\omega)\bigg{\{}C_{\xi}(\theta^{k-1}\omega)x_{i-1}(\omega,k-1)+f(\omega,k)\bigg{\}},
      n\in \mathbb{Z},
\end{equation*}
it follows from the measurability of
$\omega\mapsto f(\omega, n)$  for any fixed $n\in \mathbb{Z}$,
and the strong measurability of maps
$\omega\mapsto C_{\xi}(\theta^{k}\omega)$, $\omega\mapsto \Phi(n,\omega)$,
 $\omega\mapsto \Pi^{s}(\omega)$ and $\omega\mapsto \Phi(n,\omega)\Pi^{u}(\omega)$
 that $\omega\mapsto x_{i}(\omega,n)$ is measurable for any fixed $n\in \mathbb{Z}$.
By the induction,
we conclude that for $i\ge 0$,
$\omega\mapsto x_{i}(\omega,n)$ is measurable for any fixed $n\in \mathbb{Z}$.
This implies that  $\omega\mapsto x(\omega,n)$ is measurable for any fixed $n\in \mathbb{Z}$.
Hence,
the pair $(\ell_{1/K,\lambda}^{\infty}(\widetilde{\Omega}\times \mathbb{Z}),\ell_{1,\lambda}^{\infty}(\widetilde{\Omega}\times \mathbb{Z}))$ is
measurably properly admissible.
Moreover,
\begin{eqnarray*}
 |x(n)|
\!\!\! &\leq&\!\!\!
 \sum_{k=-\infty}^{+\infty} \|G(n-k,\theta^{k}\omega)\|\bigg{\{}\|C_{\xi}(\theta^{k-1}\omega)\|~|x(k-1)|+|f(k)|\bigg{\}}
 \\
 \!\!\! &\leq&\!\!\!
 \sum_{k=-\infty}^{+\infty}K(\theta^{k}\omega)e^{-\alpha(\omega)|n-k|}
 \|\bigg{\{}\frac{\rho(\omega)}{K(\theta^{k}\omega)}|x(k-1)|+|f(k)|\bigg{\}}
  \\
 \!\!\! &\leq&\!\!\!
 \rho(\omega){\bigg\{}\sum_{k=-\infty}^{+\infty}e^{-\alpha(\omega)|n-k|}
   e^{\lambda(\omega)(k-1)}
  {\bigg\}}
  \|x\|_{\ell_{1,\lambda(\omega)}^{\infty}(\mathbb{Z})}
 \\
 &&
 +{\bigg\{}\sum_{k=-\infty}^{+\infty}e^{-\alpha(\omega)|n-k|}
   e^{\lambda(\omega)k}
  {\bigg\}}
 \|f\|_{\ell_{1/K(\omega),\lambda(\omega)}^{\infty}(\mathbb{Z})}
 \\
\!\!\! &=&\!\!\!
 \rho(\omega)e^{\lambda(\omega)(n-1)}
 \widetilde{\gamma}_{\lambda}(\omega)
   \|x\|_{\ell_{1,\lambda(\omega)}^{\infty}(\mathbb{Z})}
  +e^{\lambda(\omega)n}\widetilde{\gamma}_{\lambda}(\omega)
  \|f\|_{\ell_{1/K(\omega),\lambda(\omega)}^{\infty}(\mathbb{Z})}.
\end{eqnarray*}
This together with (\ref{est-gamma-1}) implies that
\begin{eqnarray*}
  \|x\|_{\ell_{1,\lambda(\omega)}^{\infty}(\mathbb{Z})}
  \leq
  \frac{\widetilde{\gamma}_{\lambda}(\omega)}{1-\rho(\omega)e^{-\lambda(\omega)}\widetilde{\gamma}_{\lambda}(\omega)}
  \|f\|_{\ell_{1/K(\omega),\lambda(\omega)}^{\infty}(\mathbb{Z})}.
\end{eqnarray*}
It is clear that the positive random variable $\omega\mapsto \frac{\widetilde{\gamma}_{\lambda}(\omega)}{1-\rho(\omega)e^{-\lambda(\omega)}\widetilde{\gamma}_{\lambda}(\omega)}$
is $\theta$-invariant.

In order to prove the proper admissibility of the pair
$(\ell_{|1/K,\beta|}^{\infty}(\widetilde \Omega \times \mathbb{Z}),\ell_{|1,\beta|}^{\infty}(\widetilde \Omega \times \mathbb{Z}))$,
 for a given
  $f\in \ell_{1/K(\omega),\beta(\omega)}^{\infty}(\mathbb{Z})$ with $\omega\in \widetilde{\Omega}$,
 we consider a map
  $\Upsilon_{|1,\beta(\omega)|,f}: \ell_{|1,\beta(\omega)|}^{\infty}(\mathbb{Z})\rightarrow \ell_{|1,\beta(\omega)|}^{\infty}(\mathbb{Z})$
  defined by
\begin{equation*}
(\Upsilon_{|1,\beta(\omega)|}^{f} x)(n)
 :=
 \sum_{k=-\infty}^{+\infty} G(n-k,\theta^{k}\omega)\bigg{\{}C_{\xi}(\theta^{k-1}\omega)x(k-1)+f(k)\bigg{\}},
    \  n\in \mathbb{Z}.
\end{equation*}
Similarly to the above estimates for the map $\Upsilon_{1,\lambda(\omega)}^{f}$, we can obtain
that
 $\Upsilon_{|1,\beta(\omega)|}^{f}(\ell_{|1,\beta(\omega)|}^{\infty}(\mathbb{Z}))\subset \ell_{|1,\beta(\omega)|}^{\infty}(\mathbb{Z})$,
 and for $x_{i}\in \ell_{|1,\beta(\omega)|}^{\infty}(\mathbb{Z})$ and $i=1,2$,
\begin{eqnarray*}
 &&\|\Upsilon_{|1,\beta(\omega)|}^{f} x_{1}-\Upsilon_{|1,\beta(\omega)|}^{f} x_{2}\|_{\ell_{|1,\beta(\omega)|}^{\infty}(\mathbb{Z})}
 \\
 &\leq&
  \rho(\omega)e^{\beta(\omega)}
  \bigg{\{}\frac{1}{1-e^{-(\alpha(\omega)-\beta(\omega))}}
 +\frac{e^{-(\alpha(\omega)-\beta(\omega))}}{1-e^{-(\alpha(\omega)-\beta(\omega))}} \bigg{\}}
  \|x_{1}-x_{2}\|_{\ell_{|1,\beta(\omega)|}^{\infty}(\mathbb{Z})}.
\end{eqnarray*}
It follows from (\ref{6:07})
that $\Upsilon_{|1,\beta(\omega)|}^{f}$
is a contraction on the Banach space $\ell_{|1,\beta(\omega)|}^{\infty}(\mathbb{Z})$.
This together with Lemma \ref{lm-wt-sol} implies that the pair
$(\ell_{|1/K,\beta|}^{\infty}(\widetilde \Omega \times \mathbb{Z}),\ell_{|1,\beta|}^{\infty}(\widetilde \Omega \times \mathbb{Z}))$
is properly admissible.
Proceeding as above for the pair $(\ell_{1/K,\lambda}^{\infty}(\widetilde{\Omega}\times \mathbb{Z}),\ell_{1,\lambda}^{\infty}(\widetilde{\Omega}\times \mathbb{Z}))$, one can establish that
$(\ell_{|1/K,\beta|}^{\infty}(\widetilde \Omega \times \mathbb{Z}),\ell_{|1,\beta|}^{\infty}(\widetilde \Omega \times \mathbb{Z}))$ is measurably properly admissible.

Up to now,
we have checked all conditions of Theorem \ref{theorem-admi-texd} for $\Psi_{\xi}$ for any fixed $\xi\in \Xi$.
Hence, $\Psi_{\xi}$ has a tempered exponential dichotomy on $\widetilde{\Omega}$ for every $\xi\in \Xi$.

Furthermore,
we estimate the exponent and the bound of the tempered exponential dichotomy for $\Psi_{\xi}$
and prove the H\"older continuity of  the associated projections in $\xi$.
We need the following lemma.

\begin{lemma}
Suppose that $\Phi_{i}$  for $i=1,2$ is a measurable  linear cocycle
over $\theta$ with the generator $A_{i}$
 and has a tempered exponential dichotomy on $\widetilde{\Omega}$
with the exponent $\alpha_{i} \colon \Omega \to (0, +\infty)$ and bound $\kappa_{i}K\colon \Omega \to (0, +\infty)$,
where $\kappa_{i}\colon \Omega \to (0, +\infty) $ is $\theta$-invariant random variable
and $K$ is a tempered (with respect to $\theta$) random variable.
If  $\|A_{1}(\omega)-A_{2}(\omega)\|\leq \delta(\omega) K(\theta \omega)^{-1}$  for $\omega \in  \widetilde{\Omega}$,
  where $\delta: \Omega\rightarrow (0,+\infty)$ is a $\theta$-invariant random variable,
then for $\omega \in  \widetilde{\Omega}$,
\begin{equation}\label{wt-Green-Green-formula-lm}
 G_{2}(n,\omega)
 \!=\! G_{1}(n,\omega)\!+\!\!\sum_{k=-\infty}^{+\infty} G_{1}(n\!-\!k,\theta^{k}\omega)C_{2,1}(\theta^{k-1}\omega)G_{2}(k\!-\!1,\omega),
\end{equation}
for $n\in \mathbb{Z}$, where $G_{i}(n,\omega)$ for $i=1,2$ is the Green function corresponding to the tempered exponential dichotomy for $\Phi_{i}$
and $C_{2,1}(\omega):=A_{2}(\omega)-A_{1}(\omega)$.
 \label{lm-Green=Green}
\end{lemma}

Note that
$\{G_{2}(m,\omega)\}_{m\in \mathbb{Z}_{+}}$
and $\{\Id-G_{2}(0,\omega),G_{2}(m,\omega) {~\rm~for}~m\leq -1 \}_{m\in \mathbb{Z}_{-}}$
are bounded solutions of the equation
$x(n+1)=(\Phi_{1}(1,\theta^{n}\omega)+C_{2,1}(\theta^{n}\omega))x(n)$
on $\mathbb{Z}_{+}$ and $\mathbb{Z}_{-}$ respectively for $\omega \in  \widetilde{\Omega}$.
Then by Lemma \ref{lm-wt-sol}, the proof of (\ref{wt-Green-Green-formula-lm})
is similar to the proof of \cite[(3.40) and (3.41), p.4042]{ZhouLuZhang13JDE}.
Hence, we omit the proof of Lemma \ref{lm-Green=Green}.

By Lemma \ref{EST},
the bound of the tempered exponential dichotomy for $\Psi_{\xi}$
is of the form $\kappa_{\xi}K: \Omega\rightarrow (0,\infty)$, where $\kappa_{\xi}$ is a $\theta$-invariant random variable.
Moreover, by Lemma \ref{lm-Green=Green}  we have for $\omega \in  \widetilde{\Omega}$,
\begin{eqnarray}
 G_{\xi}(n,\omega)
 \!=\! G(n,\omega)\!+\!\!\sum_{k=-\infty}^{+\infty} G(n\!-\!k,\theta^{k}\omega)C_{\xi}(\theta^{k-1}\omega)G_{\xi}(k\!-\!1,\omega)
\label{wt-Green-Green-formula}
\end{eqnarray}
for  $n\in \mathbb{Z}$, where $G_{\xi}(n,\omega)$ denotes the Green function corresponding to the tempered exponential dichotomy for $\Psi_{\xi}$.
%
Using the discrete version of the projected Gronwall inequality
(see \cite[Corollary 2.3, p.107]{Daleckij-Krein-book-1974},
\cite[Lemma 8, p.3589]{Barreira-Silva-Valls-2009-JDE}
and
\cite[Lemma 17, p.3595]{Barreira-Silva-Valls-2009-JDE}),
one can check that the exponent $\widetilde{\alpha}:\Omega \to (0,+\infty)$ and the bound $\widetilde{K}:\Omega \to (0,+\infty)$
of the tempered exponential dichotomy for $\Psi_{\xi}$ are given by
\begin{eqnarray*}
\widetilde{\alpha}(\omega):=-\ln
(\cosh\alpha(\omega)-\{\cosh^{2}\alpha(\omega)-1-2\rho(\omega)
\sinh\alpha(\omega)\}^{1/2}),
\end{eqnarray*}
and $\widetilde{K}(\omega):=\kappa(\omega)K(\omega)$ respectively with
\[\kappa(\omega):=(1+\rho(\omega)/
\{(1-\varrho(\omega))(1-e^{-\alpha(\omega)})\})\max\{D_{1}(\omega),D_{2}(\omega)\},\]
where
\begin{eqnarray*}
\varrho(\omega)&:=&\rho(\omega) (1+e^{-\alpha(\omega)})/(1-e^{-\alpha(\omega)}),
\\
D_{1}(\omega)&:=&1/\{1-\rho(\omega) e^{-\alpha(\omega)}/(1-e^{-(\alpha(\omega)+\widetilde{\alpha}(\omega))})\},
\\
D_{2}(\omega)&:=&1/\{1-\rho(\omega) e^{-\widetilde{\beta}(\omega)}/(1-e^{-(\alpha(\omega)+\widetilde{\beta}(\omega))})\},
\\
\widetilde{\beta}(\omega)&:=&\widetilde{\alpha}(\omega)+\ln(1+2\rho(\omega) \sinh\alpha(\omega)).
\end{eqnarray*}
Hence, $\Psi_{\xi}$ has a tempered exponential dichotomy on $\widetilde \Omega$ for any $\xi\in \Xi$
with the same exponent $\widetilde{\alpha} \colon \Omega \to (0, +\infty)$ and the same
 bound $\widetilde{K}\colon \Omega \to (0, +\infty)$.

Finally, we check the H\"older continuity of the projection $\widetilde{\Pi}_{\xi}^{s}(\omega)$ in $\xi$,
where $\widetilde{\Pi}_{\xi}^{s}(\omega)$ denotes the projection of the tempered exponential dichotomy for $\Phi_{\xi}$.
For any fixed $\xi_{i}\in \Xi$ for $i=1,2$, by Lemma \ref{lm-Green=Green} we have for $\omega \in  \widetilde{\Omega}$,
\begin{eqnarray*}
 G_{\xi_{2}}(n,\omega)
 \!=\! G_{\xi_{1}}(n,\omega)\!+\!\!\sum_{k=-\infty}^{+\infty} G_{\xi_{1}}(n\!-\!k,\theta^{k}\omega)C_{\xi_{2},\xi_{1}}(\theta^{k-1}\omega)G_{\xi_{2}}(k\!-\!1,\omega),
\end{eqnarray*}
for $n\in \mathbb{Z}$, where $C_{\xi_{2},\xi_{1}}(\omega):=B_{\xi_{2}}(\omega)-B_{\xi_{1}}(\omega)$.
By the definition of the Green function, where we choose $n=0$ in the preceding formula, we get
\begin{eqnarray*}
 \widetilde{\Pi}_{\xi_{2}}(\omega)
 \!=\! \widetilde{\Pi}_{\xi_{1}}(\omega)\!+\!\!\sum_{k=-\infty}^{+\infty} G_{\xi_{1}}(-k,\theta^{k}\omega)C_{\xi_{2},\xi_{1}}(\theta^{k-1}\omega)G_{\xi_{2}}(k\!-\!1,\omega).
\end{eqnarray*}
It follows that
\begin{eqnarray*}
 &&\|\widetilde{\Pi}_{\xi_{2}}(\omega)-\widetilde{\Pi}_{\xi_{1}}(\omega)\|
 \\
 &\leq&
 \sum_{k=-\infty}^{+\infty} \|G_{\xi_{1}}(-k,\theta^{k}\omega)\| \cdot \|C_{\xi_{2},\xi_{1}}(\theta^{k-1}\omega)\| \cdot \|G_{\xi_{2}}(k\!-\!1,\omega)\|
 \\
 &\leq&
 \sum_{k=-\infty}^{+\infty} \kappa(\omega)K(\theta^{k}\omega)e^{-\widetilde{\alpha}(\omega)|k|}
 \|C_{\xi_{2},\xi_{1}}(\theta^{k-1}\omega)\|\kappa(\omega)K(\omega)e^{-\widetilde{\alpha}(\omega)|k|},
\end{eqnarray*}
because $\kappa$ is $\theta$-invariant.
This together with (\ref{Hold-cond}) implies that
\begin{eqnarray*}
 &&\|\widetilde{\Pi}_{\xi_{2}}(\omega)-\widetilde{\Pi}_{\xi_{1}}(\omega)\|
  \\
 &\leq&
  \sum_{k=-\infty}^{+\infty} \kappa(\omega)K(\theta^{k}\omega)e^{-\widetilde{\alpha}(\omega)|k|}
 \frac{\upsilon(\omega)|\xi_{2}-\xi_{1}|^{\sigma}}{K(\theta^{k}\omega)}\kappa(\omega)K(\omega)e^{-\widetilde{\alpha}(\omega)|k|}
 \\
 &=&
  \frac{\kappa(\omega)^{2}\upsilon(\omega)K(\omega)(1+e^{-2\widetilde{\alpha}(\omega)})}{1-e^{-2\widetilde{\alpha}(\omega)}}
  |\xi_{2}-\xi_{1}|^{\sigma}
\end{eqnarray*}
because $\upsilon$ is $\theta$-invariant,
This implies that
$\xi\mapsto \widetilde{\Pi}_{\xi}^{s}(\omega)$ is H\"older continuous with H\"older exponent $\sigma$.
This completes the proof.
\end{proof}

Next, we correspondingly discuss the admissible condition with no Lyapunov norms
for the deterministic difference equation (\ref{lnde-dis}).
For fixed nonnegative constants $\kappa$, $\epsilon$ and $\beta$, let
\begin{eqnarray*}
&&\ell^\infty_{1/\kappa e^{\epsilon\cdot}, \beta}(\Z):=\bigg \{f=(f(n))_{n\in \Z}\subset X: \sup_{n\in \Z} \big (\kappa e^{\epsilon |n|}
e^{-\beta n}|f(n)|\big )<+\infty \bigg \},
\\
&&\ell^\infty_{|1/\kappa e^{\epsilon\cdot}, \beta|}(\Z):=\bigg \{f=(f(n))_{n\in \Z}\subset X: \sup_{n\in \Z} \big (\kappa e^{\epsilon |n|}
e^{-\beta|n|}|f(n)|\big )<+\infty \bigg \}.
\end{eqnarray*}
Then both $\ell^\infty_{1/\kappa e^{\epsilon\cdot}, \beta}(\Z)$ and $\ell^\infty_{|1/\kappa e^{\epsilon\cdot}, \beta|}(\Z)$
are Banach spaces equipped the norms
\begin{eqnarray*}
&&
|f|_{\ell^\infty_{1/\kappa e^{\epsilon\cdot}, \beta}(\Z)}:= \sup_{n\in \Z} \big (\kappa e^{\epsilon |n|} e^{-\beta n}|f(n)|\big )
\ {\rm and}
 \\
&&
|f|_{|\ell^\infty_{1/\kappa e^{\epsilon\cdot}, \beta}(\Z)|}:= \sup_{n\in \Z} \big (\kappa e^{\epsilon |n|} e^{-\beta|n|}|f(n)|\big )
\end{eqnarray*}
 respectively.
In the particular case $\kappa=1$, $\epsilon=0$,
we use $\ell^\infty_{1, \beta}(\Z)$ and $\ell^\infty_{|1, \beta|}(\Z)$
to denote $\ell^\infty_{1/\kappa e^{\epsilon\cdot}, \beta}(\Z)$
 and $\ell^\infty_{|1/\kappa e^{\epsilon\cdot}, \beta|}(\Z)$ respectively.


\begin{cor}\label{DetmS}
Suppose that there are positive constant $\kappa$, $\epsilon$ and $\beta$ such that
the three pairs
 $(\ell^\infty_{1/\kappa e^{\epsilon\cdot}, \lambda}(\Z),\ell^\infty_{1, \lambda}(\Z))$ for
 $\lambda=\pm \beta$,
 and $(\ell^\infty_{|1/\kappa e^{\epsilon\cdot}, \beta|}(\Z),\ell^\infty_{|1, \beta|}(\Z))$
 are properly admissible for Eq.(\ref{lnde-dis}).
Then Eq.(\ref{lnde-dis}) has a nonuniform exponential dichotomy with bound $\tilde{\kappa}e^{\epsilon|n|}$
for some positive constant $\tilde{\kappa}$.
\end{cor}

We omit the detail of the proof of Corollary \ref{DetmS}.
Actually,
by Remark \ref{re-assum-gamma},
we see that the proper admissibility of the pair $(\ell^\infty_{1/\kappa e^{\epsilon\cdot}, \lambda}(\Z),\ell^\infty_{1, \lambda}(\Z))$
for $\lambda=\pm \beta$ implies that
there is a positive constant $\gamma_{\lambda}$
 such that
\begin{equation*}
\|x\|_{\ell^\infty_{1, \lambda}(\Z)}
\leq \gamma_{\lambda}\|f\|_{\ell^\infty_{1/\kappa e^{\epsilon\cdot}, \lambda}(\Z)}
\end{equation*}
for $f=(f(n))_{n\in \Z}\in \ell^\infty_{1/\kappa e^{\epsilon\cdot}, \lambda}(\Z)$,
where $x=(x(n))_{n\in \Z}$ is the unique solution of~\eqref{non-lnde-dis} in
 $\ell^\infty_{1, \lambda}(\Z)$.
Hence,
replacing $K(\theta^{n}\omega)$, $\lambda(\omega)$, $\beta(\omega)$, $A(\theta^{n}\omega)$, $U_{\omega}(n,m)$,
$L_{\lambda, \omega}$ and $G_{\lambda, \omega}(n, k)$
with
$\kappa e^{\epsilon |n|}$, $\lambda$, $\beta$, $A(n)$, $U(n,m)$,
$L_{\lambda}$ and $G_{\lambda}(n, k)$ respectively in the proof of Theorem \ref{theorem-admi-texd},
one can check that Eq.(\ref{lnde-dis}) has a nonuniform exponential dichotomy with the bound $\tilde{\kappa} e^{\epsilon|n|}$
for a constant $\tilde{\kappa}>0$.


\section{Acknowledgements}
The authors are ranked in alphabetic order of their names and their contributions should be regarded equally.
D. D  was supported in part by Croatian Science Foundation under the Project IP-2019-04-1239 and by the University of Rijeka under the Projects uniri-prirod-18-9 and uniri-prprirod-19-16.
W. Z was supported by NSFC \#11771307, \#11831012 and \#11821001.
L. Z was supported by NSFC \#12171337 and \#11771308.



\vskip 0.4cm

\noindent {\bf Appendix: Four Important Lemmas}

For convenience,
we collect the statements of the Luzin's theorem,
Poincare's reccurence theorem, Birkhoff's ergodic theorem and Kac's Lemma used in the proof of Theorem~\ref{t2}
in this appendix.

{\bf Lemma A1.}
{\bf (Luzin's theorem, \cite[Theorem 7.10, p.217]{Folland})}
{\it  Suppose that $\Omega$ is a locally compact Hausdorff space,
 $\mu$ is Radon measure on $\Omega$,
and $f: \Omega\rightarrow \mathbb{C}$ is measurable which vanishes outside a set of finite measure.
 Then for any fixed $\epsilon>0$ there exists a Borel set $E\subset \Omega$
 and a function $g\in C_{c}(\Omega)$ such that
$ f=g \ {\rm on}\ E\ {\rm and}\ \mu(\Omega\setminus E)<\epsilon,$
where $C_{c}(\Omega)$ is the class of continuous functions on $\Omega$ with compact support.
}

A Radon
measure $\mu$ on $\Omega$ is a Borel measure that is finite on all compact sets,
outer regular on all Borel sets (i.e., $\mu(E)=\inf\{\mu(U):U\supset E, U \ {\rm is \ open}\}$ for any Borel set $E$),
and inner regular on all open sets (i.e., $\mu(E)=\sup\{\mu(U):U\subset E, U \ {\rm is \ compact}\}$ for any open set $E$).
Note that a compact metric space is second countable.
Hence, by  \cite[Theorem 7.8, p.217]{Folland},
if $\Omega$ is a compact metric space then a Borel probability measure on $\Omega$ is a Radon measure.

{\bf Lemma A2.}
{\bf (Poincare's reccurence theorem, \cite[Theorem 4.1.19, p.142]{Katok-Hasselblatt-1995})}
{\it Suppose that $(\Omega, \mathcal F, \mathbb P)$ is a probability space,
$\theta:\Omega\rightarrow \Omega$ is measurable $\mathbb P$-preserving map,
and $E\in \mathcal F$.
Then
$ \mathbb P(\{\omega\in E| \{\theta^{n}\omega\}_{n\geq N}\subset \Omega\setminus E\})=0$
for any $N\in \mathbb{Z}_{+}$.
}

Under the assumptions of Poincare's reccurence theorem,
it follows that
for any fixed $E\in \mathcal F$ with $\mathbb P(E)>0$ there exists $F\subset E$ with $\mathbb P(F)=\mathbb P(E)$
such that for any $x\in F$ there exists a strict increasing sequence $\{n_{i}\}_{i}^{+\infty}\subset \mathbb{Z}_{+}$
with $\theta^{n_{i}}x\in F$.

{\bf Lemma A3.}
{\bf (Birkhoff's ergodic theorem, \cite[Theorem 4.1.2, p.136]{Katok-Hasselblatt-1995})}
{\it Suppose that $(\Omega, \mathcal F, \mathbb P)$ is a probability space,
$\theta:\Omega\rightarrow \Omega$ a measurable $\mathbb P$-preserving map
and $\varphi\in \mathcal{L}^{1}(\Omega,\mathbb P)$.
Then the limit
$$
\tilde{\varphi}(\omega):=\lim_{n\rightarrow \infty}\frac{1}{n}\sum_{k=0}^{n-1}\varphi(\theta^{k}\omega)
$$
exists $\mathbb P$-a.e. $\omega\in \Omega$ and
$\tilde{\varphi}$ is $\mathbb P$-integrable
and $\theta$-invariant
and satisfies
\begin{eqnarray*}
 \int_{\Omega}\tilde{\varphi}(\omega) d\mathbb P=\int_{\Omega}\varphi(\omega) d\mathbb P.
\end{eqnarray*}
}


{\bf Lemma A4.}
{\bf (Kac's lemma, \cite[Theorem 2', p.1006]{Kac-1947})}
\label{lm-Kac}
{\it Suppose that $f$ is a measure preserving map on a measure space $(\Omega,\mathcal{F},\mu)$ and
 let $n_{A}(x) := \inf\{k\geq 1 : f^{k}(x)\in A, x\in A\}$ for any fixed positive measure set $A\in \mathcal{F}$.
 Then
$ \int_{A}n_{A}(x)d\mu_A={1}/{\mu(A)},
$
where $\mu_{A}(\cdot):= \mu(\cdot\cap A)/\mu(A)$ denotes the normalized restriction of $\mu$ to $A$.
}



\begin{thebibliography}{11}

\bibitem[AD]{Alhalawa-Dragicevic-2017}
M. Alhalawa and D. Dragi\v cevi\' c,
On spectral characterization of nonuniform hyperbolicity,
{\it J. Funct. Spaces} (2017), Article no.6707649.

\bibitem[Ar]{Arnold}
L. Arnold,  \emph{Random Dynamical Systems}, Springer, Berlin, 1998.

\bibitem[BD]{Backes-Dragicevic-2019}
L. Backes and D. Dragi\v cevi\' c, Periodic approximation of exceptional Lyapunov exponents for semi-invertible operator cocycles,
{\it Ann. Acad. Sci. Fenn. Math.} {\bf 44} (2019), 183-209.

\bibitem[BDV1]{Barreira-Dragicevic-Valls-2014-01} L. Barreira, D. Dragi${\rm\check{c}}$evi\'c and C. Valls,
 Strong and weak $(L^{p},L^{q})$-admissibility,
 {\it Bull. Sci. Math.} {\bf 138} (2014), 721-741.

\bibitem[BDV2]{Barreira-Dragicevic-Valls-2014-02} L. Barreira, D. Dragi${\rm\check{c}}$evi\'c and C. Valls,
 Exponential dichotomies with respect to a sequence of norms and admissibility,
 {\it Inter. J.  Math.} {\bf 25} (2014), 20 pages.

\bibitem[BDV3]{Barreira-Dragicevic-Valls-2016}
L. Barreira, D. Dragi\v cevi\' c and C. Valls,
Tempered exponential dichotomies: admissibility and stability under perturbations,
{\it Dyn. Syst.} {\bf 31} (2016), 525-545.

\bibitem[BDV4]{BVDbook}
L. Barreira, D. Dragi\v cevi\' c and C. Valls,
{\it Admissibility and Hyperbolicity},
Springer, Cham, 2018.

\bibitem[BV]{Barreira-Valls-book-2008} L. Barreira and C. Valls,
{\it Stability of Nonautonomous Differential Equations},
      Lecture Notes in Math. Vol.{\bf 1926}, Springer, Berlin, 2008.

\bibitem[BSV]{Barreira-Silva-Valls-2009-JDE} L. Barreira, C. Silva  and C. Valls,
 Nonuniform behavior and robustness, {\it J. Differential Eqns.} {\bf 246} (2009), 3579-3608.

\bibitem[CLa]{Chicone-Latushkin-book-1999} C. Chicone and Y. Latushkin,
{\it Evolution Semigroups in Dynamical Systems and Differential
Equations}, Amer. Math. Soc., Providence, R.I., 1999.

\bibitem[CLe]{Chow-Leiva-1995}
S. Chow and H. Leiva, Existence and roughness of the exponential
dichotomy for skew-product semiflow in Banach spaces, {\it J.
Differential Eqns.} {\bf 120} (1995), 429-477.

\bibitem[CS]{Coffman-Schaffer-1967} C. Coffman and J. Sch\"affer,
Dichotomies for linear difference equations,
 {\it Math. Ann.} {\bf 172} (1967), 139-166.

\bibitem[Con]{Cong-book-1997} N. Cong,
{\it Topological Dynamics of Random Dynamical Systems},
 Clarendon, Oxford, 1997.

\bibitem[Cop]{Coppel-book-1978} W. Coppel, {\it Dichotomies in Stability Theory},
Lecture Notes in Math. Vol.{\bf 629}, Springer, Berlin, 1978.

\bibitem[DK]{Daleckij-Krein-book-1974} J. Daleckij and M. Krein,
 {\it Stability of Solutions of Differential Equations in Banach Space},
     Amer. Math. Soc., Providence, R.I., 1974.

\bibitem[DSS]{Dragicevic-Sasu-Sasu}
D. Dragi\v cevi\' c, A. Sasu and B. Sasu,
On the asymptotic behavior of discrete dynamical systems - An ergodic theory approach,
{\it  J. Differential Eqns.} {\bf 268} (2020), 4786-4829.

\bibitem[DLS]{Duan-Lu-Schmalfuss} J. Duan, K. Lu and B. Schmalfuss,
 Invariant manifolds for stochastic partial differential equations,
 {\it Ann. Probab.} {\bf 31} (2003), 2109-2135.

\bibitem[Fo]{Folland} G. Folland, {\it Real Analysis}, Wiley, New York, 1999.

\bibitem[GTQ]{Gonzalez-Tokman-Quas-2014}
C. Gonzalez-Tokman and A. Quas,
A semi-invertible operator Oseledets theorem,
{\it Ergodic Theory Dyn. Syst.} {\bf 34} (2014), 1230-1272.

\bibitem[Gr]{Gruendler}
J. Gruendler, Homoclinic solutions and chaos in ordinary differential equations with singular perturbations,
{\it Trans. Amer. Math. Soc.} {\bf 350}(1998), 3797-3814.

\bibitem[Gu]{Gundlach-1995} V. Gundlach,
 Random homoclinic orbits,
 {\it Random Comput. Dyn.} {\bf 3} (1995), 1-33.

\bibitem[HL]{Hale-Lin}
J. Hale and X. Lin,
Heteroclinic orbits for retarded functional differential equations,
{\it J. Differential Eqns.}  {\bf 65} (1986), 175-202.

\bibitem[Hen]{Henry-book-1981} D. Henry,
{\it Geometric Theory of Semilinear Parabolic Equations},
      Lecture Notes in Math. Vol.{\bf 840}, Springer, Berlin, 1981.

\bibitem[HM]{Huy-Minh-2001}
N. Huy and N. Minh, Exponential dichotomy of difference
equations and applications to evolution equations on the half-line,
{\it Comp. Math. Appl.} {\bf 42}(2001), 301-311.




\bibitem[Kac]{Kac-1947}
M. Kac, On the notion of recurrence in discrete stochastic processes,
{\it Bull. Amer. Math. Soc.} {\bf 53} (1947), 1002-1010.

\bibitem[KH]{Katok-Hasselblatt-1995}
A. Katok and B. Hasselblatt,
{\it Introduction to the Modern Theory of Dynamical Systems},
Cambridge Univ. Press, Cambridge, 1995.


\bibitem[LS]{LSa} Y. Latushkin and R. Schnaubelt, Evolution semigroups, translation algebra and exponential dichotomy of cocycles, {\it J. Differential Eqns.} {\bf159} (1999), 321-369.


\bibitem[LL]{Lian-Lu}
Z. Lian and K. Lu,
{Lyapunov exponents and invariant manifolds for random dynamical systems in a Banach space},
{\it Mem. Amer. Math. Soc.} {\bf 206} (2010).

\bibitem[Lin]{Lin86} X. Lin,
Exponential dichotomies and homoclinic orbits in functional differential equations,
{\it J. Differential Eqns.} {\bf 63} (1986), 227-254.

\bibitem[Mai]{Maizel-1954} A. Ma${\rm\breve{i}}$zel$'$,
 On the stability of solutions of systems of differential equations,
{\it Ural. Politehn. Inst. Trudy} {\bf 51} (1954), 20-50.

\bibitem[MS]{Massera-Schaffer-book-1966} J. Massera and J. Sch\"affer,
{\it Linear Differential Equations and Function Spaces},
Academic Press, New York, 1966.


\bibitem[MSS1]{megan-sasu-sasu-2003-DCDS} M. Megan, A. Sasu and B. Sasu,
 Discrete admissibility and exponential dichotomy for evolution families,
 {\it Discrete Contin. Dyn. Syst.} {\bf 9} (2003), 383-397.

\bibitem[MSS2]{MSS} M. Megan,  A. Sasu and B. Sasu,
Perron conditions for pointwise and global exponential dichotomy of linear skew-product flows,
{\it Integral Equations Operator Theory} {\bf 50} (2004),  489-504.

\bibitem[MZZ]{Mohammed-Zhang-Zhao} S. Mohammed, T. Zhang and H. Zhao,
{The stable manifold theorem for semilinear stochastic evolution
equations and stochastic partial differential equations},
{\it Mem. Amer. Math. Soc.} {\bf 196} (2008).


\bibitem[Osel]{Oseledets}
V. Oseledets,  A multiplicative ergodic theorem. Liapunov characteristic numbers for dynamical systems,
{\it Trans. Moscow Math. Soc.} {\bf 19} (1968), 197-221.

\bibitem[Pa1]{Palmer} K. Palmer, Exponential dichotomies and transversal homoclinic points,
{\it J. Differential Eqns.} {\bf 55} (1984), 225-256.

\bibitem[Pa2]{Palmer88} K. Palmer, Exponential dichotomies and Fredholm operators,
{\it Proc. Amer. Math. Soc.} {\bf 104} (1988), 149-156.

\bibitem[Pa3]{Palmer2}
K. Palmer, {\it Shadowing in Dynamical Systems}, Kluwer, Dordrecht, 2000.

\bibitem[Per]{Perro-1930} O. Perron, Die stabilit\"atsfrage bei differentialgleichungen,
     {\it Math. Z.} {\bf 32} (1930), 703-728.

\bibitem[Pl]{Pliss-1977} V. Pliss,
 Bounded solutions of inhomogeneous linear systems of differential equations,
 {\it Problems of the Asymptotic Theory of Nonlinear Oscillations}
(in Russian), ed. N. Bogoljubov and V. Koroljuk, Naukova Dumka, Kiev, 1977, 168-173.



\bibitem[PPC]{PPC} C. Preda, P. Preda and A. Craciunescu,
Criterions for detecting the existence of the exponential dichotomies in the asymptotic behavior of the solutions of variational equations,
{\it J. Functional Anal.} {\bf 258} (2010), 729-757.

\bibitem[Po]{C-Preda-Onofrei-2018}
C. Preda and O. Onofrei,  Discrete Sch\"affer spaces and exponential dichotomy for evolution families,
{\it Monatsh. Math.} {\bf 185} (2018), 507-523.

\bibitem[RR]{Rodr-R95} H. Rodrigues and J. Ruas-Filho,
Evolution equations: dichotomies and the Fredholm alternative for bounded solutions,
      {\it J. Differential Eqns.} {\bf 119} (1995), 263-283.

\bibitem[SSe]{Sacker-Sell-1994} R. Sacker and G. Sell,
Dichotomies for linear evolutionary equations in Banach spaces,
{\it J. Differential Eqns.} {\bf 113} (1994), 17-67.

\bibitem[Sar]{Sarig}
O. Sarig, Lecture Notes on Ergodic Theory,
available at
http://www.weizmann.ac.il/math/sarigo/sites/math.sarigo/files/uploads/ergodicnotes.pdf




\bibitem[SS1]{Sasu-Sasu-2016} A. Sasu and B. Sasu,
Admissibility and exponential trichotomy of dynamical systems described by skew-product flows,
 {\it J. Differential Eqns.} {\bf 260} (2016), 1656-1689.


\bibitem[SS2]{sasu-sasu-06-01}
B. Sasu and A. Sasu, Exponential dichotomy and $(\ell^{p},
\ell^{q})$-admissibility on the half-line, {\it J. Math. Anal. Appl.} {\bf 316} (2006), 397-408.


\bibitem[Zh]{ZWN95} W. Zhang, The Fredholm alternative and exponential dichotomies
for parabolic equations, {\it J. Math. Anal. Appl.} {\bf 191} (1995),
180-201.

\bibitem[ZLZ1]{ZhouLuZhang13JDE} L. Zhou, K. Lu and W. Zhang,
Roughness of tempered exponential dichotomies for infinite-dimensional random difference equations,
{\it J. Differential Eqns.}
{\bf 254} (2013), 4024-4046.

\bibitem[ZLZ2]{ZhouLuZhang17JDE}
L. Zhou, K. Lu and W. Zhang,
Equivalences between nonuniform exponential dichotomy and admissibility,
{\it J. Differential Eqns.}
{\bf 262} (2017), 682-747.

\bibitem[ZZ]{ZhouZhang16JFA}
L. Zhou and W. Zhang,
 Admissibility and roughness of nonuniform exponential dichotomies for difference equations,
{\it J. Functional. Anal.} {\bf 271} (2016), 1087-1129.

\bibitem[ZhuZ]{Zhu-Zhang}
C. Zhu and W. Zhang, Linearly independent homoclinic bifurcations parameterized by a small function,
{\it J. Differential Eqns.} {\bf 240} (2007), 38-57.


\end{thebibliography}
\end{document}